\documentclass[12pt]{amsart}
\usepackage{amsmath,amssymb,amsthm,
color,latexsym}
\usepackage{mathtools}
\usepackage{lineno,hyperref}
\usepackage[cp850]{inputenc}
\usepackage{wrapfig}
\usepackage{graphicx,xcolor, lineno}
\usepackage{textcomp}
\usepackage{soul}
\usepackage{cancel}
\usepackage[all]{xy} 
\usepackage{todonotes}

\newtheorem{theorem}{Theorem}[section]
\newtheorem{lemma}[theorem]{Lemma}
\newtheorem{proposition}[theorem]{Proposition}
\newtheorem{corollary}[theorem]{Corollary}

\theoremstyle{definition}
\newtheorem{definition}[theorem]{Definition}

\newtheorem{examplex}[theorem]{Example}
\newenvironment{example}
  {\pushQED{\qed}  \examplex}
  {\popQED\endexamplex}

\theoremstyle{remark}
\newtheorem{remark}[theorem]{Remark}

\numberwithin{equation}{section}

\def\tr{\mathop\mathrm{tr}\nolimits}
\def\nor{\mathop\mathrm{nor}\nolimits}

\newcommand{\restric}[2]{\left.#1\right|_{#2}}

\def\beq{\begin{equation}}
\def\eeq{\end{equation}}

\begin{document}

\title[Spacelike submanifolds into null hypersurfaces]
{Codimension two spacelike submanifolds 
into null hypersurfaces  of generalized Robertson-Walker spacetimes}

\author{Luis J. Al\'\i as}
\address{Departamento de Matem\'{a}ticas, Universidad de Murcia, E-30100 Espinardo, Murcia, Spain}
\email{ljalias@um.es}

\author{Josu\'e Mel\'endez}
\address{Departamento de Matem\'{a}ticas, Universidad Aut\'onoma Metropolitana-Iztapalapa, CP 09340,  M\'exico.} 
\email{jms@xanum.uam.mx}

\author{Matias Navarro}
\address{
Facultad de Matem\'aticas, Universidad Aut\'onoma de Yucat\'an, Perif\'erico Norte, Tablaje 13615, M\'erida, Yucat\'an, M\'exico}
\email{matias.navarro@correo.uady.mx}

\author{Didier A. Solis}
\address{Facultad de Matem\'aticas, Universidad Aut\'onoma de Yucat\'an, Perif\'erico Norte, Tablaje 13615, M\'erida, Yucat\'an, M\'exico} 
\email{didier.solis@uady.mx}

\subjclass[2010]{53C40, 53C42}

\date{May 30, 2025}

\dedicatory{To the memory of No\'e Mel\'endez Hern\'andez}

\keywords{spacelike submanifolds; Lorentz-Minkowski spacetime; de Sitter spacetime; trapped submanifolds,  light cone, cylinder over light cone,  De Sitter spacetime}

\begin{abstract}
We study codimension two spacelike submanifolds contained into a general class of null hypersurfaces in generalized Robertson-Walker spacetimes, refer to as nullcones. In particular we analyze light cones and lightlike cylinders in Lorentz-Minkowski spacetime, as well as null cones in de Sitter spacetime. We give conditions on a radial coordinate that guarantee that such a spacelike submanifold is conformally diffeomorphic to the hyperbolic space, a round cylinder  a sphere; respectively. We also provide some non-existence results for weakly trapped submanifolds. 
\end{abstract}

\maketitle

\tableofcontents

\section{Introduction}\label{s1}

Lightlike hypersurfaces play a key role in general relativity, as they realize several geometric constructions and objects related to the causal structure of spacetime\footnote{That is to say, a connected and time oriented Lorentzian manifold $(\overline{M}^{n+2}, \langle\,,\rangle)$.}. Such is the case for event, Killing and Cauchy horizons, achronal boundaries, or the conformal boundary of asymptotically flat spacetimes \cite{CHB}. Due to the Lorentzian signature of spacetime, every lightlike hypersurface $S^{n+1}\subset \overline{M}^{n+2}$ carries a  $1$ dimensional distribution in which the metric degenerates $Rad (TS)=TS\cap TS^\perp$. Therefore, the ambient metric $\langle ,\rangle$ induces a Riemannian metric in any codimension 2 smooth manifold $\Sigma^n\subset S^{n+1}$ whose tangent bundle $TS$ is complementary to $Rad(TM)$ in $TM$. For globally hyperbolic spacetimes $(\overline{M}^{n+2}, \langle\,,\rangle)$, such manifolds arise naturally as cross sections of a horizon when intersected with a Cauchy surface. It is in these context where the concept of marginally trapped surface enters as a key hypothesis on the celebrated singularity theorems of Penrose \cite{BH,penrose}. Furthermore, some of the main accomplishments of black hole theory deal with the characterization of possible topological types of the cross sections of the event horizon, dating back to the seminal work of Hawking on the spherical topology of black hole horizons for stationary spacetimes \cite{ER,Galloway,Hawking}.

From the geometrical point of view, we say that a Riemannian manifold $(\Sigma^n,\langle ,\rangle)$ {\em factors through a lightlike hypersurface} $S^{n+1}$ if there is an isometric immersion $\psi :\Sigma\to \overline{M}$ such that $\psi (\Sigma )\subset S$.  Several results pertaining the rigidity of marginally trapped $\Sigma$ and their topological type have been established recently in different contexts, as Lorentz-Minkowski space \cite{ACR,PPR}, Lorentzian space forms \cite{ACR2,CFP}, Robertson Walker spacetimes \cite{ACC} and gravitational waves \cite{CPR}. 

 In \cite{GO}, the authors introduce the notion of {\em null cones} and use to it characterize  totally umbilical lightlike hypersurfaces in Robertson-Walker spacetimes (see also \cite{NPS2016} for related results). In this work we develop a framework in order to analyze the geometry of submanifolds $\Sigma$ that factor through a null cone. As applications we analyze the non compact case and show that under mild assumptions $\Sigma$ is conformally diffeomorphic to hyperbolic space. Moreover, we establish the non existence of weakly trapped $\Sigma$. We also study codimension two spacelike submanifolds of Lorentz-Minkowski space that factor through a lightlike cylinder and show they are conformally diffeomorphic to a spherical cylinder.
 
 This paper is organized as follows. In Section \ref{sec:prelim} we establish the notation that will be used throughout the paper and set up our general framework. Then, in Section \ref{sec:examples} we analyze the particular cases of spacetimes of the form $\overline{M}=-I\times M^{n+1}$ and $\overline{M}=-i\times_f\mathbb{R}^{n+1}$. This is further specialized to the cases of spacelike submanifolds that factor through the future component and the lightcone in Section \ref{sec:LC}, and  a cylinder over a lightcone in Section \ref{sec:cylinder}. As a last application, we study spacelike submanifolds that factor through nullcones in De Sitter space.

\section{Preliminaries}\label{sec:prelim}

\subsection{Lightlike hypersurfaces in generalized Robertson-Walker spacetimes}

Let $(M, \langle\,,\rangle_M)$ be a Riemannian manifold and $I$ an open interval of $\mathbb{R}$. We consider the ($n + 2$)-dimensional generalized Robertson-Walker (GRW) spacetime\footnote{Recall that a Robertson-Walker spacetime is a GRW spacetime where the fiber $M$ has constant sectional curvature.}
\[
\overline{M}^{n+2}=-I\times_f M^{n+1}, 
\]
endowed with the Lorentzian warped metric
 \[
 \langle\,,\rangle= -dt^2+f^2(t) \ \langle\,,\rangle_M.
 \]
Further, we choose on $\overline{M}$ the time-orientation given by the globally defined timelike unit vector field
\[
\partial_t = 
\restric{(\partial/\partial t)}{(t,x)} \in \mathfrak{X}(\overline{M}),
\]
for every $(t,x) \in \overline{M}$.

All lightlike hypersurfaces in GRW spacetimes can be realized as graphs of transnormal functions associated to a (signed) distance to a fixed submanifold \cite[Prop. 3.4]{NPS2016}. We consider here the particular case of a singleton. Thus,  we define the function
$F:\overline{M}^{n+2} \to \mathbb{R}$
by
\[
F(t,x)=-\left( \int_{t_0}^t \frac{ds}{f(s)} \right)^2 + r^2(x), 
\]
where $p_0=(t_0,x_0)\in \overline{M}$ and $r(x)=d_M(x_0,x)$ denotes the Riemannian distance function from the fixed point $x_0\in M$.  If $0$ is a regular value, we define the {\em future nullcone at} $p_0$ as the  hypersurface\footnote{Compare \eqref{eq:lightcones} to the local form of null cones as described in \cite[Prop. 3.1]{GO}.} 
\begin{equation}\label{eq:lightcones}
 \Lambda^+=\{(t,x)\in \overline{M}\,:\,F(t,x)=0, t>t_0\}.
\end{equation}

The gradient $\overline{\nabla}F \in \mathfrak{X}(\overline{M})$ of $F$ decomposes as  
\[
\overline{\nabla}F(t,x)=2\varphi(t) \ \partial_t + { U}(t,x)
\]
where ${ U}$ is the tangent component to the fiber $M$ and $\varphi$ is given by
\[
\varphi(t)
=
-\dfrac{1}{2}\langle \overline{\nabla}F, \partial_t\rangle=-\dfrac{1}{2}\dfrac{\partial F}{\partial t}(t,x)=\dfrac{1}{f(t)}\left( \int_{t_0}^t \dfrac{ds}{f(s)}\right). \]
Let ${ V}\in \mathfrak{X}(M)$, hence the lift ${\hat{V}} \in \mathfrak{L}(M) $ of ${ V}$ to $\overline{M}$ satisfies 
\[
{\hat{V}}(F)=\langle \overline{\nabla}F,{\hat{V}}\rangle =f^2(t)\langle { U},{ V} \rangle_M.
\]
On the other hand, 
\[
{ \hat{V}}(F)={V}(r^2(x)) =2r(x)\langle { V}, Dr(x) \rangle_M
\]
where $D$ denotes the  gradient operator on $M$. Consequently,
\[
2r(x) Dr(x) = f^2(t) { U},
\]
which implies
\[
{ U}(t,x)=\dfrac{2r(x)}{f^2(t)}Dr(x).
\]
Thus 
\begin{equation}\label{eq:gradF}
\frac{1}{2}\overline{\nabla}F(t,x)=\varphi(t)\partial_t+\dfrac{r(x)}{f^2(t)} Dr(x)=
\dfrac{1}{f(t)}\left( \int_{t_0}^t \dfrac{ds}{f(s)}\right) \partial_t + \dfrac{r(x)}{f^2(t)} Dr(x)
\end{equation}
and
\[
    \frac{1}{4}\langle \overline{\nabla}F, \overline{\nabla}F \rangle = -\dfrac{1}{f^2(t)}\left( \int_{t_0}^t \dfrac{ds}{f(s)}\right)^2 + \dfrac{r^2(x)}{f^4(t)}\langle Dr,Dr \rangle 
    =\dfrac{1}{f^2(t)}F(t,x). 
\]
Therefore,  
\begin{equation}\label{eq:null_lambda}
\langle \overline{\nabla}F(t,x), \overline{\nabla}F(t,x) \rangle = 0,
\quad \text{for all $(t,x)\in \Lambda^+$}. 
\end{equation}
Moreover,
\[
\langle \overline{\nabla}F, X \rangle = 0 \quad \text{for every   $X\in \mathfrak{X}(\Lambda^+)$} .
\]
It follows that $\overline{\nabla}F$ is a lightlike future-pointing vector field, and hence $\Lambda^+$ is a lightlike hypersurface with tangent space given by
\[
T_{(t,x)} \Lambda^+ = \{ {\bf v} \in T_{(t,x)} \overline{M}\,:\,\langle \overline{\nabla}F(t,x), {\bf v} \rangle =0 \},
 \qquad (t,x) \in \Lambda^+.
\]

\begin{example}\label{ex:L}
If $M=\mathbb{R}^{n+1}$ and $x_0=0$, we get  $\overline{M} =-I\times_f \mathbb{R}^{n+1}$. In this case,  $r(x)=\Vert x\Vert$ with $Dr(x)=\dfrac{x}{\Vert x\Vert}$. Hence, we have the expression 
\[
\frac{1}{2}\overline{\nabla}F(t,x)=\dfrac{1}{f(t)}\left( \int_{t_0}^t \dfrac{ds}{f(s)}\right) \ \partial_t + \dfrac{x}{f^2(t)},
\qquad t>t_0. 
\]
In particular,  taking $f=1$, $I=\mathbb{R}$ and $t_0=0$, we have $\overline{M}$ is nothing but the ($n+2$)-dimensional Lorentz-Minkowski spacetime $\mathbb{L}^{n+2}$ and $\Lambda^+$ is the future component of light cone,
\[
\Lambda^+= \{ 
(t,x) \in \mathbb{L}^{n+2} : -t^2+ \Vert x \Vert ^2 =0, \,t>0\}.\]
In this case,
\[
F(t,x)= -t^2 + \Vert x\Vert^2
\qquad
\text{and}
\qquad
 \frac{1}{2}\overline{\nabla}F(t,x)=t\, \partial_t + x.
\]
\end{example}

\subsection{Submanifolds that factor through a nullcone}

Let us consider an isometric immersion $\psi:\Sigma^n \to \overline{M}=-I\times_f M^{n+1}$  of an $n$-dimensional connected Riemannian manifold $\Sigma$ with $\psi (\Sigma^n)\subset \Lambda^+$. In that case, we will say the $\Sigma$ factors through the nullcone $\Lambda^+$. Hence
\[
\xi(p)=\frac{1}{2}\overline{\nabla}F(\psi(p))=\frac{1}{2}\overline{\nabla}F(t(p),x(p))
\]
is a smooth future lightlike vector field everywhere orthogonal to $\Sigma$.

Denote by $\pi_I:\overline{M} \to I$ the projection $\pi_I(t,x)=t$ and by $\pi_M:\overline{M} \to M$ the projection $\pi_M(t,x)=x$. We define the {\em timelike coordinate function} and the {\em spacelike coordinate function} of the immersion $\psi:\Sigma \to \overline{M}$ by
\[
u=\pi_I \circ \psi:\Sigma \to I
\]
and
\[
\phi=\pi_M \circ \psi:\Sigma \to M.
\]

Then the  lightlike vector field $\xi\in T^\perp (\Sigma)$ satisfies
\begin{equation}\label{eq:xid}
\xi=\dfrac{1}{f(u)}\left( \int_{t_0}^u \dfrac{ds}{f(s)}\right) \ \partial_t \mid_{(u,\phi)}+\dfrac{r(\phi)}{f^2(u)} Dr(\phi).
\end{equation}
Consider now the expression
\[
\partial_t=\partial_t^\top + \partial_t^\perp,
\]
where $\partial_t^\top\in \mathfrak{X}(\Sigma)$ and $\partial_t^\perp\in\mathfrak{X}^\perp(\Sigma)$. In particular, 
\[
\langle\partial_t^\perp, \partial_t^\perp\rangle=-1-\Vert \partial_t^\top \Vert^2 \leq -1<0.
\]
Note that 
\[
\nu=\dfrac{\partial_t^\perp}{\Vert \partial_t^\perp \Vert}= \frac{\partial_t^\perp}{\sqrt{1+\Vert \partial_t^\top \Vert^2}}
\]
determines a globally defined future-pointing unit timelike vector field along  $\Sigma$. Since
\[
\overline{\nabla}\pi_I=-\partial_t
\quad \text{and}
\quad 
u=\pi_I\circ \psi,
\]
we have that
\begin{equation}\label{eq:toppartialt}
\nabla u = (\overline{\nabla} \pi_I)^\top = -\partial_t^\top.
\end{equation}
It follows that
\[
\nu=\dfrac{\partial_t^\perp}{\sqrt{1+\Vert \partial_t^\top \Vert^2}}=\dfrac{1}{\sqrt{1+\Vert \nabla u \Vert^2}}\left( \partial_t - \partial_t^\top \right) = \dfrac{1}{\sqrt{1+\Vert \nabla u \Vert^2}}\left( \partial_t + \nabla u \right).
\]
Finally, let us define
\[
\eta = -\dfrac{1}{2\langle \xi, \nu \rangle^2} \xi - \dfrac{1}{\langle \xi, \nu \rangle} \nu.
\]
Then we have
\[
\langle \xi, \xi \rangle =  \langle \eta, \eta \rangle = 0, \quad \langle \xi, \eta \rangle = -1,
\]
and
\[
    \langle \xi, \nu \rangle = \langle \frac{1}{2} \overline{\nabla}F, \nu \rangle 
    = \dfrac{1}{2\sqrt{1+\Vert \nabla u \Vert^2}}\langle \overline{\nabla}F, \partial_t \rangle 
    = -\dfrac{1}{f(u)\sqrt{1+\Vert \nabla u \Vert^2}}\int_{t_0}^u \dfrac{ds}{f(s)}<0.
\]

The following result summarizes our findings.
\begin{proposition}\label{prop:null_frame}
Let $\psi:\Sigma^n \to  \overline{M}$ be a codimension two spacelike submanifold of $\overline{M}$ that factors through $\Lambda^+$. Then there exists a globally defined future-pointing normal  null frame $\{\xi, \eta \}$  given by
\begin{equation*}
\xi=\frac{1}{f(u)}\left( \int_{t_0}^u \frac{ds}{f(s)}\right)\partial_t + \dfrac{r(\phi)}{f^2(u)}Dr(\phi) 
\end{equation*}
and 
\begin{equation*}
\eta=-\dfrac{f^2(u)(1+\Vert \nabla u \Vert^2)}{2\left( \int_{t_0}^u \frac{ds}{f(s)} \right)^2}\xi + \dfrac{f(u)}{\int_{t_0}^u \frac{ds}{f(s)}}\partial_t^\perp,
\end{equation*}
and satisfies $\langle\xi, \eta\rangle=-1$. 
\end{proposition}

In particular, and as a direct application of Proposition \ref{prop:null_frame}, we  extend \cite[Prop. 3.1]{ACR}.

\begin{corollary}
\label{coro:prop:null_frame}
Let $\psi:\Sigma^n \to \Lambda^+ \subset \mathbb{L}^{n+2}$ be a codimension two spacelike submanifold which is contained in the future component of the light cone $\Lambda^+$ of the Lorentz-Minkowski spacetime. Let $u= -\langle \psi, {\bf e_{1}}\rangle$. Then,
\begin{align*}
\xi =\frac{1}{2}\overline{\nabla}F(\psi ) =  \psi
\qquad \text{and}
\qquad
\eta=-\dfrac{(1+\Vert \nabla u \Vert^2)}{2 u^2}\xi + \dfrac{1}{u} {\bf e_1}^\perp.   \end{align*}
\end{corollary}

Recall that $\displaystyle{\varphi(t) = \frac{1}{f(t)} \int_{t_0}^t \frac{ds}{f(s)}}$ and $\displaystyle{\frac{1}{2}\overline{\nabla}F(t,x)=\varphi(t)\partial_t+\dfrac{r(x)}{f^2(t)} Dr(x)}$.
For a vector field $Z \in T{\overline{M}}$, consider   
\[
Z=-\langle Z, \partial_t \rangle  \partial_t + \hat{Z}, \qquad \hat{Z} = (\pi_M)_*(Z). 
\]
Thus, we get
\begin{align*}
\frac{1}{2}\overline{\nabla}_Z\overline{\nabla}F = -\frac{1}{2}\langle Z, \partial_t \rangle  \overline{\nabla}_{\partial_t}\overline{\nabla}F + \frac{1}{2}
\overline{\nabla}_{\hat{Z}}\overline{\nabla}F 
\end{align*}
where, 
$\overline{\nabla}$ denotes the Levi-Civita connection of $\overline{M}$. Since $\overline{\nabla}_{\partial_t}\partial_t=0$  and $\overline{\nabla}_{\partial_t}Dr=\cfrac{\partial_t(f)}{f}Dr$, we have
\begin{eqnarray}
\frac{1}{2}\overline{\nabla}_{\partial_t}\overline{\nabla}F&=&\varphi'(t) \partial_t +  \partial_t \left( \frac{r}{f^2(t)}\right) Dr + \frac{r}{f^2(t)}\overline{\nabla}_{\partial_t} Dr \label{a-part1}\\ 
&=&\left( -\frac{f'(t)}{f^2(t)} \left( \int_{t_0}^{t} \frac{ds}{f(s)} \right) + \frac{1}{f^2(t)} \right) \partial_t - \frac{rf'(t)}{f^3(t)}Dr.\nonumber
\end{eqnarray}
As $\hat{Z}$ is tangent to $M$ we also have
\[
\frac{1}{2}\overline{\nabla}_{\hat{Z}}\overline{\nabla}F  = \overline{\nabla}_{\hat{Z}}(\varphi(t)\partial_t)+\overline{\nabla}_{\hat{Z}}\left(\frac{1}{f^2(t)}rDr \right) 
= \varphi(t) \overline{\nabla}_{\hat{Z}}\partial_t + \frac{1}{f^2(t)} \overline{\nabla}_{\hat{Z}}(r Dr).
\]
Notice that
\begin{equation}
\overline{\nabla}_{\hat{Z}}\partial_t=\overline{\nabla}_{\partial_t}\hat{Z}=\frac{\partial_t(f)}{f} \hat{Z}=\frac{f'(t)}{f(t)} \hat{Z},
\end{equation}
and
\begin{eqnarray*}
\overline{\nabla}_{\hat{Z}}(r Dr)&=&\nor(\overline{\nabla}_{\hat{Z}}(rDr) ) + \tan(\overline{\nabla}_{\hat{Z}}  ( r Dr ) ) \\
&=&-\frac{\langle \hat{Z}, r Dr \rangle}{f(t)} \overline{\nabla}f + \nabla_{\hat{Z}}^M (r Dr) \\
&=&\frac{f'(t) r \langle \hat{Z}, Dr \rangle}{f(t)} \partial_t + \nabla_{\hat{Z}}^M(r Dr),
\end{eqnarray*}
where   $\nor v$ and $\tan v$ denote, respectively, the orthogonal and tangent  projections of $v \in T_{(t,x)} {\overline{M}}$ on $T_{(t,x)} (I \times \{ x\} )$, and $\overline{\nabla}f=-f'(t) \ \partial_t$.

As $\hat{Z}=Z + \langle Z, \partial_t \rangle \partial_t$, we then have 
\begin{eqnarray}\label{b-part1}
\frac{1}{2}\overline{\nabla}_{\hat{Z}}\overline{\nabla}F & = & \frac{\varphi(t)f'(t)}{f(t)} \hat{Z} + \frac{1}{f^2(t)}\left( \frac{1}{f(t)} f'(t) r(x) \langle \hat{Z}, Dr \rangle \partial_t + \nabla_{\hat{Z}}^M (r Dr)  \right) \nonumber \\
	&=&\frac{f'(t)}{f^2(t)}\left( \int_{t_0}^{t} \frac{ds}{f(s)}\right) \hat{Z} + \frac{f'(t) r}{f^3(t)} \langle Z, Dr \rangle \partial_t + \frac{1}{f^2(t)} \nabla_{\hat{Z}}^M (r Dr)\nonumber \\
	&=&\left( \frac{f'(t)}{f^2(t)}\langle Z, \partial_t \rangle \left(\int_{t_0}^{t}\frac{ds}{f(s)} \right) + \frac{f'(t) r}{f^3(t)} \langle Z, Dr \rangle \right) \partial_t \\
& &+\frac{f'(t)}{f^2(t)}\left( \int_{t_0}^{t} \frac{ds}{f(s)} \right)Z + \frac{1}{f^2(t)}\nabla_{\hat{Z}}^M(r Dr). \nonumber
\end{eqnarray}

Thus by combining \eqref{a-part1} and \eqref{b-part1} we obtain, for any $Z$ tangent to $\overline{M}$,
\begin{eqnarray}\label{nablaxi}
\frac{1}{2}\overline{\nabla}_{Z}\overline{\nabla}F & = & \frac{\langle Z, \partial_t \rangle}{f^2(t)}\left(2f'(t)\left( \int_{t_0}^{t} \frac{ds}{f(s)} \right) - 1 \right) \partial_t \nonumber \\
& &+\frac{f'(t)r}{f^3(t)}\left( \langle Z, Dr \rangle \partial_t + \langle Z, \partial_t \rangle Dr \right)  \\
& &+\frac{f'(t)}{f^2(t)}\left( \int_{t_0}^{t}\frac{ds}{f(s)} \right)Z + \frac{1}{f^2(t)} \nabla_{\hat{Z}}^M(r Dr). \nonumber
\end{eqnarray}

\section{The cases $\overline{M}=-I\times M^{n+1}$ and $\overline{M}=-I\times_fM^{n+1}$.}\label{sec:examples}

In this section we work out in detail the form of the shape operators $A_\xi$ and $A_\eta$ associated to $\Sigma$ and derive some general consequences for the cases of a Lorentzian product spacetime $\overline{M}=-I\times M^{n+1}$ and for a GRW spacetime of the form $\overline{M}=-I\times_fM^{n+1}$.

At this point, it is worth pointing out that throughout all this paper we are using for the second fundamental form $II$ and for the shape operator $A_\eta$ associated to a normal vector field $\eta$ the standard convention in General Relativty, and opposite to the one usually used in Differential Geometry. Therefore, for a general spacelike submanifold $\Sigma$ in an arbitrary Lorentzian ambient space $\overline M$, the Gauss and Weingarten formulas are given, respectively, by
\begin{equation}
\label{Gauss_luis}
\overline\nabla _XY=\nabla_XY-II(X,Y)
\end{equation}
and 
\begin{equation}
\label{Wein_luis}
\overline\nabla _X\eta=A_\eta X+\nabla^\perp_X\eta
\end{equation}
for every $X,Y\in T\Sigma$ and $\eta\in T^\perp\Sigma$. 

First we consider  the case  $f\equiv 1$  which corresponds to the immersion in the Cartesian product 
\[
\psi:\Sigma \to \Lambda^+\subset \overline{M}=-\mathbb{R}\times M^{n+1}.
\]
In this case, $F(t,x)$ is given simply by 
\[
F(t,x)=-t^2+r(x)^2
\]
and, choosing $t_0=0$, (\ref{eq:gradF}) becomes
\[
\frac{1}{2}\overline{\nabla}F(t,x)=t\partial_t+r(x)Dr(x).
\]
Then  \eqref{nablaxi} reduces to
\begin{equation}
\label{luis1}
    \frac{1}{2}\overline{\nabla}_Z\overline{\nabla}F=-\langle Z, \partial_t \rangle \partial_t + \nabla_{\hat{Z}}^M(r Dr),
\end{equation}
for every $Z$ tangent to $\overline{M}$. 
Letting $t_0=0$. from Proposition \ref{prop:null_frame} we obtain along the submanifold the null frame
\[
\xi=\frac{1}{2}\overline{\nabla}F(\psi)=u \partial_t + r Dr
\qquad
\text{and}
\qquad
\eta=-\frac{1+\Vert \nabla u \Vert^2}{2u^2}\xi+\frac{1}{u}\partial_t^{\perp},
\]
and from (\ref{eq:toppartialt}) and (\ref{luis1}) their corresponding shape operators are given by
\begin{equation*}
    A_\xi X=(\overline \nabla_X\xi)^{\top}=\frac{1}{2}(\overline{\nabla}_X\overline{\nabla}F)^\top=
    -\langle X, \nabla u\rangle \nabla u + (\nabla^M_{\hat{X}}(r Dr))^{\top},
\end{equation*}
and 
\begin{equation*}
A_\eta X
=
(\overline \nabla_X\eta)^{\top}
=
-\frac{1+\Vert \nabla u \Vert^2}{2u^2}A_\xi X+\frac{1}{u}A_{\partial_t^{\perp}}X,
\end{equation*}
for every tangent vector field $X\in T\Sigma$. 

We now proceed to compute $A_{\partial_t^\perp}$. This computation is independent from the Riemannian factor $M^{n+1}$ and from the warping function $f$, so that it can be done it its full generality for a GRW spacetime, as showed in the next result.
\begin{lemma}
	\label{lem:shape}
	Let $\psi:\Sigma^n \to\Lambda^{+}\subset\overline{M}=-I\times_fM^{n+1}$ be a codimension two spacelike submanifold that factors through the nullcone $\Lambda^{+}$. Then,
	\begin{equation*}
	A_{\partial_t^\perp} X
	=     
	\nabla_X \nabla u +
	\frac{f'(u)}{f(u)} \left( 
	X+ \langle X, \nabla u \rangle \nabla u \right),
	\qquad 
	X\in T\Sigma.    
	\end{equation*}
	In particular, $$\theta_{\partial_t^\perp} =\frac{1}{n} \tr A_{\partial_t^\perp} = 
	\frac{1}{n} \Delta u
	+
	\frac{f'(u)}{nf(u)}
	\left(n+\|\nabla u\|^2 \right).$$
\end{lemma}

\begin{proof}
	By the standard formula for warped products
	\begin{equation}
	\label{eq:U}
	\overline{\nabla}_Z \partial_t = \frac{f'(t)}{f(t)}(Z+\langle Z,\partial_t\rangle\partial_t)
	\end{equation}
	for every $Z \in T\overline{M}$. In particular, if $X\in T\Sigma$ and using
	$\partial_t=\partial_t^\top+\partial_t^\perp=-\nabla u+\partial_t^\perp$, 
	this reduces to
	\[
	\overline{\nabla}_X \partial_t=
	\frac{f'(u)}{f(u)}(X+\langle X,\nabla u\rangle\nabla u)-\frac{f'(u)}{f(u)}\langle X,\nabla u\rangle\partial_t^\perp.
	\]
	On the other hand, using Gauss (\ref{Gauss_luis}) and Weingarten (\ref{Wein_luis}) formulas we have
	\begin{equation*} 
	\overline{\nabla}_X \partial_t  = 
	\overline{\nabla}_X\partial_t^\top+
	\overline{\nabla}_X\partial_t^\perp
	= 
	-\nabla_X\nabla u+II(X,\nabla u)
	+ A_{\partial_t^\perp} X
	+ \nabla_X^{\perp}\partial_t^\perp,
	\end{equation*}
	from which we conclude
	\begin{equation}
	\label{luis5}
	A_{\partial_t^\perp} X=\frac{f'(u)}{f(u)} \left( 
	X+ \langle X, \nabla u \rangle \nabla u \right)+\nabla_X\nabla u
	\end{equation}
	and
	\begin{equation}
	\label{luis6}
	\nabla_X^\perp\partial_t^\perp=-\frac{f'(u)}{f(u)}\langle X,\nabla u\rangle\partial_t^\perp-II(X,\nabla u).
	\end{equation}
	This completes the proof of our lemma.
\end{proof}
As a consequence, we obtain the following result.
\begin{proposition}
	\label{prop:shapeoperators_product}
	Let $\psi:\Sigma^n \to \Lambda^+\subset\overline{M}=-\mathbb{R}\times M^{n+1}$ be a codimension two spacelike submanifold of a Lorentzian product spacetime that factors through te nullcone. Then, for every $X\in T\Sigma$ we have
	\begin{equation*}
	A_\xi X =  -\langle X, \nabla u\rangle \nabla u + (\nabla^M_{\hat{X}}(r Dr))^{\top}
	\end{equation*}    
	and
	\begin{equation*}
	A_\eta X = \frac{1+\|\nabla u\|^2}{2u^2}\left(\langle X, \nabla u\rangle \nabla u - (\nabla^M_{\hat{X}}(r Dr))^{\top}\right)+\frac{1}{u}\nabla_X\nabla u
	\end{equation*}
	where $\hat{X}=X-\langle X,\nabla u\rangle\partial_t$.  
\end{proposition}

As an immediate consequence, when $M^{n+1}=\mathbb{R}^{n+1}$ is the Euclidean space, using that $rDr(x)=x$ for every $x\in\mathbb{R}^{n+1}$ we get
\[
(\nabla^{\mathbb{R}^{n+1}}_{\hat{X}}(rDr))^\top=(\hat{X})^\top=X+\langle X, \nabla u \rangle\nabla u,
\]
and we find the Weingarten endomorphisms associated to the future light cone in Lorentz-Minkowski spacetime, in agreement with \cite[Prop. 3.2]{ACR}. 
\begin{corollary}
	\label{coro:prop:shapeoperators}
	Let $\psi:\Sigma^n \to \Lambda^+ \subset \mathbb{L}^{n+2}$ be a codimension two spacelike submanifold which factors through the future component of the light cone $\Lambda^+$ of the Lorentz-Minkowski spacetime. Let $u= -\langle \psi, {\bf e_{1}}\rangle$. Then,
	\begin{equation*}
	A_\xi  =   I
	\qquad
	\text{and}
	\qquad 
	A_\eta = 
	- \frac{(1+\|\nabla u\|^2)}{2u^2}I  
	+ 
	\frac{1}{u} 
	\nabla^2 u  
	\end{equation*}
	where $\nabla^2 u (X) =\nabla_X \nabla u$.
	In particular,
	\begin{equation*}
	\theta_\xi = \frac{1}{n} \tr (A_\xi)= 1
	\qquad
	\text{and}
	\qquad
	\theta_\eta
	=
	\frac{1}{n} \tr (A_\eta) 
	=
	\frac{-n(1+\|\nabla u\|^2) + 2u  \Delta u}{2 n u^2}.
	\end{equation*}
\end{corollary}

For our next application,  we let the manifold $M$ be the Euclidean space $\mathbb{R}^{n+1}$ and consider an arbitrary warping function $f:I\to \mathbb{R}^+$, that is to say, $\overline{M}=-I\times_f \mathbb{R}^{n+1}$. Choosing $x_0=0\in\mathbb{R}^{n+1}$, $r(x)=\|x\|^2$ and $F(t,x)$ is given by 
\[
F(t,x)=-\left( \int_{t_0}^t \dfrac{ds}{f(s)}\right)^2+\|x\|^2,
\]
with $r(x)Dr(x)=x$ and 
\[
\frac{1}{2}\overline{\nabla}F(t,x)=\dfrac{1}{f(t)}\left( \int_{t_0}^t \dfrac{ds}{f(s)}\right) \partial_t + \dfrac{1}{f^2(t)} x.
\]
In this case, 
\begin{eqnarray*}
	Dr &=&\frac{x}{r},\\
	\nabla^{\mathbb{R}^{n+1}}_{\hat{Z}}(rDr) &=& \nabla^{\mathbb{R}^{n+1}}_{\hat{Z}}(x)=\hat{Z} = Z+\langle Z, \partial_t \rangle \partial_t,
\end{eqnarray*}
and, after a direct computation, \eqref{nablaxi} takes the form
\begin{eqnarray}
	\label{luis4}
\nonumber	\frac{1}{2}\overline{\nabla}_Z\overline{\nabla}F & = & \frac{f'(t)}{f(t)} \left(   \left\langle  Z,  \overline{\nabla} F \right\rangle -  \frac{1}{f(t)^2}  \langle Z,x  \rangle \right) \partial_t \\
	{} & {} & + \frac{f'(t)}{f^3(t)} \langle Z, \partial_t \rangle x +\frac{1}{f^2(t)} \left( f'(t)\left( \int_{t_0}^t \frac{ds}{f(s)} \right) + 1 \right) Z.
\end{eqnarray}

Let $\psi:\Sigma\rightarrow\Lambda^{+}\subset \overline{M}=-I\times_f\mathbb{R}^{n+1}$ the spacelike immersion, with $\psi=(u,\phi)$, for which
\begin{equation}
\label{luis3}
F(\psi)=-\left( \int_{t_0}^u \dfrac{ds}{f(s)}\right)^2+\|\phi\|^2=0 \Leftrightarrow
\textcolor{black}{\|\phi\|=\int_{t_0}^u \dfrac{ds}{f(s)}>0}.
\end{equation}
Therefore, in this case,
\begin{equation}
\label{luis2}
\xi=\frac{1}{2}\overline\nabla F(\psi)=\dfrac{1}{f(u)}\left( \int_{t_0}^u \dfrac{ds}{f(s)}\right) \partial_t + \dfrac{1}{f^2(u)}\phi=
\frac{\|\phi\|}{f(u)}\partial_t+\frac{1}{f^2(u)}\phi
\end{equation}
and from (\ref{luis4}) 
\begin{eqnarray*}
	\overline\nabla_X\xi & = & \frac{1}{2}\overline{\nabla}_X\overline{\nabla}F=\frac{f'(u)}{f^3(u)} \left( -\langle X,\phi  \rangle\partial_t +\langle X, \partial_t \rangle \phi \right) \\
	{} & {} & +\frac{1}{f^2(u)} \left( f'(u)\left( \int_{t_0}^u \frac{ds}{f(s)} \right) + 1 \right) X.
\end{eqnarray*}
for every tangent vector field $X\in T\Sigma$, since \textcolor{black}{$\langle X,\overline{\nabla}F\rangle=0$}. In particular, from the Weingarten equation 
\begin{equation*}
\overline{\nabla}_{X} \xi 
= A_\xi X + \nabla^\perp_X \xi, \qquad X \in T\Sigma,
\end{equation*}
it immediately follows
\[
A_{\xi}X=(\overline\nabla_X\xi)^\top=\frac{f'(u)}{f^3(u)}
(-\langle X,\phi^\top\rangle\partial_t^\top+\langle X,\partial_t^\top\rangle\phi^\top)+\frac{1}{f^2(u)} \left( f'(u)\left( \int_{t_0}^u \frac{ds}{f(s)} \right) + 1 \right)X
\]
and
\[
\nabla^\perp_X{\xi}=(\overline\nabla_X\xi)^\perp=
\frac{f'(u)}{f^3(u)}
(-\langle X,\phi^\top\rangle\partial_t^\perp+\langle X,\partial_t^\top\rangle\phi^\perp).
\]
Since $\xi^\top=0$, from (\ref{luis2}) we get
\[
\phi^\top=-f(u)\|\phi\|\partial_t^\top=f(u)\|\phi\|\nabla u,
\]
so that, 
\[
-\langle X,\phi^\top\rangle\partial_t^\top+\langle X,\partial_t^\top\rangle\phi^\top=
f(u)\|\phi\|\langle X,\nabla u\rangle\nabla u-f(u)\|\phi\|\langle X,\nabla u\rangle\nabla u=0,
\]
which, jointly with (\ref{luis3}) gives
\[
A_{\xi}X=\frac{1}{f^2(u)} \left( f'(u)\left( \int_{t_0}^u \frac{ds}{f(s)} \right) + 1 \right)X=
\frac{1}{f^2(u)} \left( f'(u)\|\phi\|+ 1 \right)X
\]
for every $X\in T\Sigma$. 

On the other hand, in order to obtain the expression of $A_\eta$
we observe from Proposition \ref{prop:null_frame} and \ref{luis3} that
\begin{eqnarray*}
A_\eta & = & -\dfrac{f^2(u)(1+\Vert \nabla u \Vert^2)}{2\left( \int_{t_0}^u \frac{ds}{f(s)} \right)^2}A_\xi + \dfrac{f(u)}{\int_{t_0}^u \frac{ds}{f(s)}}A_{\partial_t^\perp}\\
{} & = & -\dfrac{f^2(u)(1+\Vert \nabla u \Vert^2)}{2\|\phi\|^2}A_\xi + \dfrac{f(u)}{\|\phi\|}A_{\partial_t^\perp},
\end{eqnarray*}
where $\textcolor{black}{A_{\partial^\perp_t}}$ is given by Lemma \ref{lem:shape}. As a consequence, we can give the following result.
\begin{proposition}
\label{prop:shapeoperators}
Let $\psi:\Sigma^n \to \Lambda^+ \subset -I\times_f \mathbb{R}^{n+1}$ be a codimension two spacelike submanifold of $\overline{M}$ that factors through the nullcone. Then, for every $X\in T\Sigma$
\begin{equation*}
A_\xi X =  \frac{f'(u)\|\phi\|+1}{f^2(u)} X
\end{equation*}    
and
\[
A_\eta X=-\left(\frac{1+\|\nabla u\|^2}{2\|\phi\|^2}+\frac{f'(u)(\|\nabla u\|^2-1)}{2\|\phi\|}\right)X
+ \frac{f'(u)}{\|\phi\|}\langle X, \nabla u \rangle \nabla u
+ \frac{f(u)}{\|\phi\|} \nabla_X \nabla u,
\]
where $\|\phi\|=\int_{t_0}^u \frac{ds}{f(s)}$, $u>t_0$. In particular,
\begin{equation*}
\theta_\xi = \frac{1}{n} \tr (A_\xi) = \frac{f'(u)\|\phi\|+1}{ f^2(u)}    
\end{equation*}
and
\begin{equation*}
\theta_\eta=\frac{1}{n} \tr (A_\eta) 
=\frac{-n(1+\|\nabla u\|^2)+2f(u)\|\phi\|\Delta u}{2n\|\phi\|^2}+
\frac{f'(u)}{2n\|\phi\|}(n-(n-2)\|\nabla u\|^2).
\end{equation*}
\end{proposition}
The proof of Proposition \ref{prop:shapeoperators} follows directly from our previous computations.

\section{Submanifolds that factor through the  light cone}\label{sec:LC}

In this section we establish some results pertaining complete codimension $2$ spacelike submanifolds that factor through the future lightcone \textcolor{black}{of the Lorentz-Minkowski space} that are conformally diffeomorphic to the hyperbolic space. Furthermore, we study some obstructions to the existence of (weakly) trapped $\Sigma^n$ for $2\le n\le 4$. Our analysis is based on the growth of a radial function, and complements previous works that analyze the compact case via a height function (see \cite{ACR}).

We denote by $\mathbb{L}^{n+2}$ the ($n + 2$)-dimensional Lorentz-Minkowski spacetime, and $\Lambda^+$ the future component of the light cone, as described in Example \ref{ex:L}.  We recall that the $n$-dimensional hyperbolic space $\mathbb{H}^n$ is the complete simply connected Riemannian manifold with sectional curvature $-1$,
\[
\mathbb{H}^n=\{x \in \mathbb{L}^{n+1} \,| \, \langle x, x\rangle =-1, \, x_1>0 \}.
\]

First, we state a couple of technical Lemmas. The following result was proved in  \cite[Lemma 5.1]{ACR} (see also \cite[Lemma 5.2]{AM})

\begin{lemma}
\label{lem:complete}
Let $g$ be a complete metric on a Riemannian manifold $\Sigma$ and let $r$ denote the Riemannian distance function from a fixed origin $o\in \Sigma$. If a function $\lambda \in C^{\infty}(\Sigma)$ satisfies 
\beq
\label{hypLem}
\lambda^{2/(n-2)}(p)  \geq \frac{C }{ r(p) \log(r(p))} , \qquad r(p)\gg1,
\eeq 
$C$ a positive constant, then the conformal metric $\tilde{g}=\lambda^{4/(n-2)} g$ is also complete.
\end{lemma}

Moreover, in \cite[Chap. VIII, Lemma 8.1]{KN} we find the next auxiliary lemma:
\begin{lemma} 
\label{lem:coveringmap}
Let $f$ be a map from a connected complete Riemannian manifold $\Sigma$ onto another connected Riemannian manifold $M$ of the same dimension. If $f$ increases the distance, that is, 
\[
\| {\bf v} \| \leq \| df_p ({\bf v})\|
\]
for all $p\in \Sigma$ and ${\bf v} \in T_p \Sigma$.
Then $f$ is a covering map and $M$ is also complete.  
\end{lemma}

\textcolor{black}{Let $\psi:\Sigma^n \to\Lambda^+\subset \mathbb{L}^{n+2}$ be a codimension two spacelike submanifold that factors through the future lightcone $ \Lambda^+$. Under the existence of a non-vanishing spacelike coordinate function, our first main result provides appropriate bounds on the growth of that function that guarantees that $\Sigma$ is conformally diffeomorphic to the hyperbolic space.}

\begin{theorem}\label{thm:02}
Let $\psi:\Sigma^n \to\Lambda^+\subset \mathbb{L}^{n+2}$ be a codimension two  complete spacelike submanifold that factors through the future lightcone $ \Lambda^+$. Assume that  \textcolor{black}{there exists a non-vanishing spacelike coordinate} function $w=\psi_i=\langle \psi, {\bf e}_{i}\rangle$. with $2\leq i\leq n+2$, satisfying 
\begin{equation} \label{hyp:1bis}
0<w(p) \leq C r(p) \log(r(p)), \qquad r(p)\gg1 ,
\end{equation}
where $C$ is a positive constant and $r$ denotes the Riemannian distance function from a fixed origin $o \in \Sigma$. Then 
$\Sigma$ is conformally diffeomorphic to the hyperbolic space $\mathbb{H}^{n}$.
\end{theorem}
\begin{proof} \textcolor{black}{Without loss of generality, we can assume that $\psi_{n+2}\neq 0$, and --by further applying an isometry if necessary-- that $\psi_{n+2}=\langle \psi ,{\bf e}_{n+2}\rangle >0$.}  

For every $p\in \Sigma$, $\psi(p)=(\psi_1(p), \psi_{2}(p), \ldots, \psi_{n+1}(p), w(p)) \in \Lambda^+$. Note that
\[
\frac{1}{w^2(p)} \left(-\psi_1^2(p) + \psi_2^2(p)+\cdots+\psi_{n+1}^2(p)\right)=-1.
\] 
Define the function $\Psi: \Sigma^n \to 
\mathbb{H}^{n}$ by
\[
\Psi(p)= \left( \frac{\psi_1(p)}{w(p)}, \ldots,   \frac{\psi_{n+1}(p)}{w(p)} \right).
\]
For every $p\in \Sigma$ and ${\bf v} \in T_p \Sigma$ we have
\[
d\Psi_p({\bf v})
=
\frac{1}{w(p)} \left({\bf v}(\psi_1), \ldots, {\bf v}(\psi_{n+1})\right) 
-
\frac{{\bf v}(w)}{w^2(p)} \left(\psi_1(p), \ldots,\psi_{n+1}(p)\right).
\]

Denote by $\langle\,,\rangle_{\mathbb{H}^n}$ the standard metric of the hyperbolic space $\mathbb{H}^{n}$. Hence, for every ${\bf v, w} \in T_p\Sigma$ we have  
\begin{eqnarray*}
\langle d\Psi_p({\bf v}), d\Psi_p({\bf w})\rangle_{\mathbb{H}^n} &=&
 -\frac{{\bf v}(\psi_1) {\bf w}(\psi_1)}{w^2(p)}
+\frac{1}{w^2(p)}\sum_{i=2}^{n+1}{\bf v}(\psi_i) {\bf w}(\psi_i)  
 +\frac{{\bf v}(w) {\bf w}(w) }{ w^2(p)}\\
&=& 
\frac{1}{w^2(p)} \langle d\psi_p({\bf v}),  d\psi_p({\bf w}) \rangle_{\mathbb{L}^{n+2}}\\
&  =&  \frac{1}{w^2(p)} \langle {\bf v},  {\bf w} \rangle_{\Sigma},
\end{eqnarray*}
where $\langle \,,\rangle_{\Sigma}$ denotes the Riemannian metric on $\Sigma$ induced by the immersion $\psi$.  Thus, 
\begin{equation} \label{iso1bis}
\Psi^*(\langle\,,\rangle_{\mathbb{H}^n}) = \frac{1}{w^2} \langle\,,\rangle_\Sigma ,
\end{equation}
thus establishing that $\Psi$ is a local diffeomorphism.

Assume now that $\Sigma$ is complete and $w$ satisfies \eqref{hyp:1bis}. Hence, by applying Lemma \ref{lem:complete} to the function $\lambda=w^{-(n-2)/2}$, we conclude that the conformal metric 
\[
\widetilde{\langle\,,\rangle} = \frac{1}{w^2}\langle\,,\rangle_\Sigma
\]
is also complete on $\Sigma$. Then, equation \eqref{iso1bis} implies that the map
\[
\Psi:(\Sigma^n, \widetilde{\langle\,,\rangle}) \to (\mathbb{H}^n, \langle\,,\rangle_{\mathbb{H}^n})
\]
is a local isometry from a (connected)  complete  Riemannian manifold $\Sigma$ to a Riemannian manifold $\Psi(\Sigma)$.  By Lemma \ref{lem:coveringmap}, it follows $\Psi$ is a covering map and $\Psi(\Sigma)$ is complete. In particular, this implies 
\[
\Psi(\Sigma)= \mathbb{H}^{n}.
\]
Finally, since $\mathbb{H}^{n}$ is simply connected, we conclude that $\Psi$ is, in fact, a global diffeomorphism between $\Sigma$ and the hyperbolic space  $\mathbb{H}^{n}$. 
\end{proof}

The following is an immediate consequence of Theorem  \ref{thm:02}.
\begin{corollary}
Let $\psi:  \Sigma^n \to \Lambda^+ \subset \mathbb{L}^{n+2}$ be a codimension two complete spacelike submanifold that factors through the light cone $\Lambda^+$.  If \textcolor{black}{there exists a non-vanishing spacelike coordinate} function $w=\psi_i=\langle \psi ,{\bf e}_{i}\rangle$, with $2\leq i\leq n+2$, satisfying \eqref{hyp:1bis}, then 
 $\Sigma^n$ cannot be compact.
\end{corollary}

\begin{example} \label{ex:hyperbolic}
Given a positive smooth function $f: \mathbb{H}^n \to (0,+\infty)$ we can construct an embedding $\psi_f : \mathbb{H}^n \to \mathbb{L}^{n+2}$ by setting
\[
\psi_f(p)= (f(p) p , f(p)).
\]
We note that
\[
\langle \psi_f(p),\psi_f(p)\rangle = f^2(p)  \langle  p ,  p \rangle_{\mathbb{H}^n} + f^2(p) =0
\] and $u(p) = -\langle \psi_f, {\bf e}_1\rangle = f(p) p_1 >0$.  Hence $\psi_f( \mathbb{H}^n)$ is contained in  $\Lambda^+$.  

For every $p\in \mathbb{H}^n$, ${\bf v}\in T_p \mathbb{H}^n$, we obtain
\[
d(\psi_f)_{p}({\bf v})= ({\bf v}(f) p + f(p) {\bf v}, {\bf v}(f)).
\]
Finally notice
\[
\langle d(\psi_f)_{p}({\bf v}),d(\psi_f)_{p}({\bf w})\rangle_{\mathbb{L}^{n+2}}= f^2(p) \langle {\bf v},{\bf w}\rangle_{\mathbb{H}^n}.
\]
This implies that $\psi_f$ determines a spacelike immersion of $\mathbb{H}^n$ into $\Lambda^+$ whose induced metric is conformal to the standard metric of the hyperbolic space. 
\end{example}

The previous example is generic in virtue of our next result, which establishes that every codimension two \textcolor{black}{complete} spacelike submanifold contained in $\Lambda^+$ \textcolor{black}{and satisfying \eqref{hyp:1bis} for a non-vanishing spacelike coordinate function} is --up to a conformal diffeomorphism-- as described above.

\begin{theorem}  
Let $\psi:  \Sigma^n \to \Lambda^+ \subset \mathbb{L}^{n+2}$ be a codimension two complete spacelike submanifold that factors through the light cone $\Lambda^+$. Assume that \textcolor{black}{there exists a non-vanishing spacelike coordinate function $w=\psi_i=\langle \psi ,{\bf e}_{i}\rangle$, with $2\leq i\leq n+2$, satisfying \eqref{hyp:1bis}}. Then there exists a conformal diffeomorphism $\Psi: (\Sigma^n, \langle \,,\rangle_\Sigma) \to  (\mathbb{H}^n, \langle \,,\rangle_{\mathbb{H}^n }) $ and  a function $f : \mathbb{H}^{n} \to (0,+\infty)$ such that 
\[
\psi = \psi_f\circ \Psi 
 \]
 where   $\psi_f: \mathbb{H}^n \to \Lambda^+ \subset \mathbb{L}^{n+2}$ is the embedding 
 \[
\psi_f(p)= (f(p) p , f(p)).
\]
In particular, the immersion $\psi$ is an embedding.
\[
\xymatrix{
\Sigma^n  \ar[r]^-{w}  & (0,+\infty)\\
 \mathbb{H}^n \ar[ru]_-{f}   \ar[u]^-{\Psi^{-1}} 
} 
\qquad  
\qquad 
\textcolor{black}{\xymatrix{
\Sigma^n  \ar[r]^-{\psi} \ar[d]_-{\Psi} & \Lambda^{+}\\
 \mathbb{H}^n \ar[ru]_-{\psi_f}   
}}
\]
\end{theorem}

\begin{proof} Assume, without loss of generality, that $w=\psi_{n+2}>0$. The first part of the proof follows the same reasoning as in Theorem \ref{thm:02}, using the \textcolor{black}{conformal diffeomorphism} $\Psi: \Sigma^n \to  \mathbb{H}^{n}$ given by
\[
\Psi= \left( \frac{\psi_1}{w},  \frac{\psi_{2}}{w} ,  \ldots, \frac{\psi_{n+1}}{w} \right) .
\] 
Hence, by setting 
\[
\textcolor{black}{f=w\circ \Psi^{-1}} 
\]
we have
\[
(\psi_f\circ \Psi )(p)=(f(\Psi(p) ) \Psi(p) , f(\Psi(p) ))=(\psi_1(p), \ldots,\psi_{n+1}(p), w(p))=\psi(p),
\]
thus completing the proof.
\end{proof}

In the remainder of this section we study the problem of existence of (weakly) trapped submanifolds that factors through the light cone. Let us recall   the globally defined future-pointing normal null frame $\{\xi,\eta\}$ as  described in Corollary \ref{coro:prop:null_frame}.

The original formulation of trapped surfaces was given by Penrose \cite{penrose} in terms of the signs or the vanishing of the null expansions. Following this approach, for a codimension two spacelike submanifold in \textcolor{black}{$\Lambda^+\subset\mathbb{L}^{n+2}$} we have the following:
\begin{enumerate}
\item[(a)] $\Sigma$ is a trapped submanifold if and only if either both $\theta_\xi < 0$ and $\theta_\eta < 0$ (future trapped), or both $\theta_\xi > 0$ and $\theta_\eta > 0$ (past trapped).
\item[(b)] $\Sigma$ is a marginally trapped submanifold if and only if either $\theta_\xi = 0$ and $\theta_\eta\neq 0$ (future marginally trapped if $\theta_\eta < 0$ and past marginally trapped if $\theta_\eta > 0$), or $\theta_\xi \neq 0$ and $\theta_\eta = 0$ (future marginally trapped if $\theta_\xi < 0$ and past marginally trapped if $\theta_\xi > 0$).
\item[(c)] $\Sigma$ is a weakly trapped submanifold if and only if either both $\theta_\xi \leq  0$ and $\theta_\eta \leq 0$ with $\theta_\xi^2 + \theta_\eta^2 >0$ (future weakly trapped), or both $\theta_\xi \geq 0$ and $\theta_\eta \geq 0$ with $\theta_\xi^2 + \theta_\eta^2 > 0$ (past weakly trapped).
\end{enumerate}

Recall that by Corollary \ref{coro:prop:shapeoperators} the null expansions are given by
\begin{equation*}
\theta_\xi = \frac{1}{n} \mathop\mathrm{tr} A_\xi =1 \qquad \text{and}\qquad \theta_\eta = \frac{1}{n} \mathop\mathrm{tr} A_\eta = \frac{2 u \Delta u- n\textcolor{black}{(1+\|\nabla u\|^2)}}{2n u^2},
\end{equation*}
where $u= -\langle \psi, {\bf e_{1}}\rangle$. Then the mean curvature vector of $\Sigma$ satisfies 
\begin{equation}
\label{eq:mcv}
{\bf H} = -\theta_\xi \eta  -\theta_\eta \xi = -\frac{1}{2nu^2} (2u\Delta u-n\textcolor{black}{(1+\|\nabla u\|^2)} \xi - \eta.    
\end{equation}
Therefore,
\begin{equation}
\label{eq:mcv^2}
\langle {\bf H}, {\bf H} \rangle  =- 2 \theta_\xi \theta_\eta = -\frac{1}{nu^2} (2 u \Delta u-n( 1+ \|\nabla u\|^2)). 
\end{equation}

Hence, from \eqref{eq:mcv} and \eqref{eq:mcv^2}, we assert that
\begin{enumerate}
\item[(a)] $\Sigma$ is trapped if and only if either ${\bf H}$ is timelike and future-pointing (future trapped) or ${\bf H}$ is timelike and past-pointing (past trapped).
\item[(b)] $\Sigma$ is marginally trapped if and only if either ${\bf H}$ is null and future-pointing (future marginally trapped) or ${\bf H}$ is null and past-pointing (past marginally trapped).
\item[(c)] $\Sigma$  is weakly trapped if and only if either ${\bf H}$ is causal and future-pointing (future weakly trapped) or ${\bf H}$ is causal and past-pointing (past weakly trapped).
\end{enumerate}

As a consequence of \ref{eq:mcv^2}, we obtain the following result (see  \cite[Corollary 6.2]{ACR}) 
\begin{corollary}
\label{corollary:trapped}
Let $\psi: \Sigma \to \Lambda^+ \subset \mathbb{L}^{n+2}$ be a codimension two spacelike submanifold that factors through the light cone $\Lambda^+$. Let $u$ be the positive function 
$u =-\langle \psi, {\bf e_1}  \rangle$. Then
\begin{enumerate}
\item[(a)] $\Sigma$ is (necessarily past) trapped if and only if $u$ satisfies the differential inequality
$$2 u \Delta u-n( 1+ \|\nabla u\|^2) >0 \qquad \text{on $\Sigma$}.$$
\item[(b)] $\Sigma$ is (necessarily past) marginally trapped if and only $u$ satisfies the differential equation
$$2 u \Delta u-n( 1+ \|\nabla u\|^2) =0 \qquad \text{on $\Sigma$}.$$
\item[(c)] $\Sigma$ is (necessarily past) weakly trapped if and only if $u$ satisfies the differential inequality
$$2 u \Delta u-n( 1+ \|\nabla u\|^2) \geq 0\qquad \text{on $\Sigma$}.$$
\end{enumerate}
\end{corollary}


\begin{remark}
\label{remark_scal}
Using the Gauss equation of the immersion $\psi: \Sigma \to \Lambda^+ \subset \mathbb{L}^{n+2}$, it easily follows from \eqref{eq:mcv^2} that the scalar curvature of $\Sigma$ is given by
\begin{equation}
 \label{eq:scal-u}
\mathop\mathrm{Scal}=n(n-1)\langle {\bf H},{\bf H} \rangle=-\frac{ 1}{u^2} (n-1)(2 u \Delta u-n( 1+ \|\nabla u\|^2)).
\end{equation}
As a consequence, the immersion $\psi: \Sigma \to \Lambda^+ \subset \mathbb{L}^{n+2}$ is 
\begin{enumerate}
\item[(a)] (necessarily past) trapped if and only if $\mathrm{Scal}<0$
on $\Sigma$.
\item[(b)] (necessarily past) marginally trapped if and only $\mathrm{Scal}\equiv 0$ on $\Sigma$.
\item[(c)] (necessarily past) weakly trapped if and only if $\mathrm{Scal}\leq 0$ on $\Sigma$.
\end{enumerate}
\end{remark}

As a direct consequence of Remark \ref{remark_scal}, we observe that there exists no weakly trapped immersion through the light cone of $\mathbb{L}^{n+2}$ of a Riemannian manifold $\Sigma^n$ with non-negative scalar curvature which is not scalar-flat. More generally, when $n=2$ we can show the following non-existence result about weakly trapped conformal immersions. Before stating our result, recall that a (non necessarily complete) Riemannian manifold $\Sigma^n$ is said to be parabolic if the only subharmonic functions on $\Sigma$ which are bounded from above are the constant functions.
\begin{theorem}
Let $(\Sigma^2,  \langle\,,\rangle)$ be a non-flat parabolic Riemannian surface with non-negative Gaussian curvature. 
Then there exists no codimension two weakly trapped immersion 
 \[
 \psi: \Sigma^2 \to \Lambda^+ \subset \mathbb{L}^{4}
 \]
that factors through the light cone $\Lambda^+$ for which  $\Sigma^2$ is a conformal surface, that is, $\psi^*(\langle\,,\rangle_{\mathbb{L}^{4}})=\lambda^2 \langle\,,\rangle$, where the conformal factor $\lambda$ is bounded from above. In particular, this holds if $(\Sigma^2,  \langle\,,\rangle)$ is a non-flat complete Riemannian surface with non-negative Gaussian curvature.
\end{theorem}

\begin{proof} 
Let $\widetilde{\langle\,,\rangle}=\lambda^2 \langle\,,\rangle$ and $\psi: (\Sigma^2,\widetilde{\langle\,,\rangle}) \to \Lambda^+ \subset \mathbb{L}^{4}$ be a weakly trapped surface such that $\psi(\Sigma^2) \subset \Lambda^+$.  From Remark \ref{remark_scal}, we see that the Gaussian curvature  $\widetilde{K}$ of $(\Sigma^2,\widetilde{\langle\,,\rangle})$ is non-positive.

On the other hand, from \eqref{eq:gauss_conformal} 
we know that the Gaussian curvatures $K$ and $\widetilde{K}$ of the conformal metrics $\langle\,,\rangle$ and $\widetilde{\langle\,,\rangle}=\lambda^2 \langle \,,\rangle$, respectively, satisfy 
\[
\lambda^2 \widetilde{K} =  K - \Delta \log(\lambda) \leq 0,
\]
where $\Delta$ denotes the Laplacian with respect to the metric $\langle\,,\rangle$.  From our hypothesis on $K$ we have
\begin{equation} \label{ineq:K-lambda}
 0\leq K \leq   \Delta \log(\lambda) \qquad \text{with} \qquad \sup_\Sigma \log(\lambda) < + \infty.
 \end{equation}
In other words, the function $\log(\lambda)$ is a subharmonic function on $(\Sigma^2,\langle\,,\rangle)$, which is bounded from above. Since $(\Sigma^2,\langle\,,\rangle)$ is parabolic, it then follows from \eqref{ineq:K-lambda} that $\log(\lambda)$ must be a constant function and $K=0$ on $\Sigma$, which is a contradiction.

In particular, by a classical result of  Ahlfors \cite{A} and Blanc-Fiala-Huber \cite{Huber}, every complete Riemannian surface with non-negative Gaussian curvature is parabolic. Therefore, the result holds true if $(\Sigma^2,  \langle\,,\rangle)$ is a non-flat complete Riemannian surface with non-negative Gaussian curvature.
\end{proof}

The proof of our next result is based also on the parabolicity of the manifold.
\begin{theorem}
 \label{thm:parab}
 Let $(\Sigma^n,  \langle\,,\rangle )$ be a parabolic Riemannian manifold with non-negative scalar curvature and dimension $2\leq n\leq 4$. There exists no codimension two trapped immersion 
 $\psi: \Sigma^n \to \Lambda^+ \subset \mathbb{L}^{n+2}$
for which $\psi^*(\langle\,,\rangle_{\mathbb{L}^{n+2}})=\lambda^2 \langle\,,\rangle$, 
where the conformal factor $\lambda$ is bounded from above.
\end{theorem}
\begin{proof}
Let $\psi:(\Sigma^n,\widetilde{\langle\,,\rangle}) \to \Lambda^+ \subset \mathbb{L}^{n+2}$ be a codimension two trapped immersion for which $\psi^*(\langle\,,\rangle_{\mathbb{L}^{n+2}})=\widetilde{\langle\,,\rangle}=\lambda^2 \langle\,,\rangle$. In particular, from Remark \ref{remark_scal} we know that the scalar curvature $\widetilde{\mathrm{Scal}}$ of the conformal metric $\widetilde{\langle\,,\rangle}$ satisfies 
\begin{equation}
\label{eq:scaltilde}
\widetilde{\mathrm{Scal}}< 0 \qquad \text{on } \Sigma.
\end{equation}

From Lemma \ref{lem:conformal}, the scalar curvature $\mathrm{Scal}$ of $\langle\,,\rangle$ and $\widetilde{\mathrm{Scal}}$ are related by
\begin{equation}
\label{eq:scal2}
\lambda^2 \, \widetilde{\mathop\mathrm{Scal}} = \mathop\mathrm{Scal} -2(n-1) \frac{\Delta\lambda}{\lambda}- (n-1)(n-4) \frac{\|\nabla\lambda\|^2}{ \lambda^2},
\end{equation}
where $\Delta, \nabla$ and $\|\cdot\|$ denotes the Laplacian, the gradient and the norm with respect to the metric $\langle\,,\rangle$. Recall also that 
\begin{equation}
\label{eq:scal3}
\mathop\mathrm{Scal}\geq 0 \qquad \text{on } \Sigma.
\end{equation}
It follows from \eqref{eq:scaltilde}, \eqref{eq:scal2} and \eqref{eq:scal3} that
\begin{equation}
\label{ineq:scal}
0\leq \mathop\mathrm{Scal}< 2(n-1) \frac{\Delta\lambda}{\lambda} 
+
(n-1)(n-4) \frac{\|\nabla\lambda\|^2}{ \lambda^2}.
\end{equation}
As $2\leq n\leq 4$, we obtain that 
\[
2 \frac{\Delta\lambda}{\lambda} \geq (4-n) \frac{\|\nabla\lambda\|^2}{ \lambda^2}\geq 0.
\]
Hence, $\Delta\lambda \geq 0$ with $\sup_{\Sigma } \lambda<+\infty$, and, by the parabolicity of $(\Sigma,\langle\,,\rangle)$, $\lambda$ is a constant. Then, from \eqref{ineq:scal} and \eqref{eq:scal2} we get $\mathop\mathrm{Scal}=\widetilde{\mathop\mathrm{Scal}}=0$ on $\Sigma$,
which is not possible because of \eqref{eq:scaltilde}. This finishes the proof.
\end{proof}

Our next result makes use of the stochastic completeness of the manifold. Recall that a (non necessarily complete) Riemannian manifold $\Sigma$ is said to be stochastically complete if for some (and therefore, for any) $(x,t)\in\Sigma\times(0,+\infty)$ it holds that $\int_\Sigma p(x,y,t)dy=1$, where $p(x,y,t)$ is the heat kernel of the Laplacian operator. As proved by Pigola, Rigoli and Setti \cite{PRS1} (see also \cite[Theorem 3.1]{PRS}), a Riemannian manifold $\Sigma$ is stochastically complete if and only if the weak Omori-Yau maximum principle holds on $\Sigma$, in the following sense: for every $u\in\mathcal{C}^2(\Sigma)$ with $u^*=\sup_\Sigma<+\infty$, there 
exists a sequence $\{p_k\}\subset\Sigma$ such that 
\[
u(p_k)>u^*-\frac{1}{k} \quad \text{and} \quad \Delta u(p_k)<\frac{1}{k} 
\]
for every $k\geq 1$.

 \begin{theorem}
Let $(\Sigma^n, \langle\,,\rangle )$ be a stochastically complete Riemannian manifold with $S_{*}:=\inf_\Sigma \mathrm{Scal}>0$ and dimension $2\leq n\leq 4$.
There exists no codimension two weakly trapped immersion 
 $\psi: \Sigma^n \to \Lambda^+ \subset \mathbb{L}^{n+2}$ for which
$\psi^*(\langle\,,\rangle_{\mathbb{L}^{n+2}})=\lambda^2 \langle\,,\rangle$,
where the conformal factor $\lambda$ is bounded from above.
\end{theorem}
\begin{proof}
Let $\psi: (\Sigma^n, \widetilde{\langle\,,\rangle}) \to \Lambda^+ \subset \mathbb{L}^{n+2}$ be a codimension two weakly trapped immersion such that 
$\psi^*(\langle\,,\rangle_{\mathbb{L}^{n+2}})=\widetilde{\langle\,,\rangle}=\lambda^2 \langle\,,\rangle$. In particular, $\widetilde{\mathop\mathrm{Scal}}\leq 0$. Reasoning as at the beginning of the
proof of Theorem \ref{thm:parab}, we can obtain
\[
0<S_*\leq\mathop\mathrm{Scal} \leq 2(n-1) \frac{\Delta \lambda}{\lambda} 
+
(n-1)(n-4) \frac{\|\nabla \lambda\|^2}{ \lambda^2}\leq 2(n-1) \frac{\Delta \lambda}{\lambda},
\]
since $2\leq n \leq 4$.

Suppose that $\lambda$ is bounded from above, that is, $\lambda^* = \sup_{\Sigma^n} \lambda < +\infty$. In addition, we note that $\lambda^*>0$. Since $(\Sigma, \langle\,,\rangle)$ is
stochastically complete, by the weak maximum principle there exists a sequence
$\{p_k\}_{k\geq 1} \subset \Sigma^n$ such that
\[
\lambda(p_k) > \lambda^*- \frac{1}{k}
\qquad
\text{and}
\qquad
\Delta\lambda(p_k) <\frac{1}{k}
\] 
for every $k\geq 1$. Then 
\[
0<S_* \leq \mathop\mathrm{Scal} (p_k) \leq 2(n-1) \frac{\Delta\lambda (p_k)}{\lambda(p_k)}  <\frac{2(n-1) }{ \lambda(p_k)} \cdot \frac{1}{k}.
\]
Since $\lambda (p_k) \to \lambda^*>0$,  from here it follows  $\mathop\mathrm{Scal} (p_k)  \to 0$ as $k\to +\infty$,  which is a contradiction.
\end{proof}

\section{Submanifolds through the lightlike cylinder \textcolor{black}{$\Lambda^+\times\mathbb{R}$}.}\label{sec:cylinder}

Among the distinguished lightlike hypersurfaces of Lorentz-Minkowski space we encounter those which posses the highest degrees of symmetry. In \cite{NPS1} all lightlike revolution hypersurfaces in Lorentz-Minkowski spacetime are characterized as either lightlike planes, light cones or lightlike cylinders. In this section we analyze this latter class of hypersurfaces. Thus, we consider spacelike submanifolds $\Sigma$ that factors through a cylinder of the form
\[
\Lambda^+ \times \mathbb{R} = \{ (x,y)\in \mathbb{L}^{n+1}\times \mathbb{R}= \mathbb{L}^{n+2} \,:\, \textcolor{black}{\langle x,x \rangle_{\mathbb{L}^{n+1}}=0}, \, x_1 >0,  \, y \in \mathbb{R}\} ,
\]
where $\Lambda^+$ is the future component of the light cone of the Lorentz-Minkowski spacetime $\mathbb{L}^{n+1}$. For another approach to this study, we refer the reader to the recent paper \cite{GYT} where the authors develop a completely different and independent investigation of codimension-two spacelike submanifolds through the lightlike cylinder $\Lambda^+\times\mathbb{R}$. In particular, in \cite{GYT} the authors analyze the geometric interpretations of the associated extrinsic invariants by providing characterizations of (pseudo-)isotropic and pseudo-umbilical submanifolds, as well as local classification of these classes of submanifolds.

Our first result in this context is the following.
\begin{theorem}\label{thm:01}
Let $\psi:  \Sigma^n \to\textcolor{black}{\Lambda^+\times\mathbb{R}\subset} \mathbb{L}^{n+2}$ be a codimension two complete spacelike submanifold that factors through $\Lambda^+ \times \mathbb{R}$  with $n\geq 3$. Assume that the positive function $u= -\langle \psi, {\bf e}_1 \rangle $ satisfies 
\begin{equation}
\label{hyp:1}
u(p) \leq C r(p) \log(r(p)), \qquad r(p)\gg1,
\end{equation}
where $C$ is a positive constant and $r$ denotes the Riemannian distance function from a fixed origin $o \in \Sigma$. Let $w = \langle \psi, {\bf e}_{n+2}\rangle.$
If $\displaystyle{ \frac{dw}{u}}$ is an exact form  on $\Sigma$, then 
$\Sigma$ is conformally diffeomorphic to the cylinder $\mathbb{S}^{n-1} \times \mathbb{R} \subset \mathbb{R}^{n+1}$.
\end{theorem}
\begin{proof}For every $p\in \Sigma$, let $\psi(p)=(u(p), \psi_{2}(p), \ldots, \psi_{n+1}(p), w(p))$, where 
\[
\sum_{i=2}^{n+1} \psi_i^2(p)=u^2(p)>0.
\] 
Further, let $g: \Sigma \to \mathbb{R}$ be  a smooth function such that $dg= \displaystyle{ \frac{dw}{u}}$. Then
define the function $\Psi: \Sigma^n \to \Psi(\Sigma) \subset \mathbb{S}^{n-1} \times \mathbb{R}$ by
\[
\Psi(p)= \left( \frac{\psi_2(p)}{u(p)}, \ldots,   \frac{\psi_{n+1}(p)}{u(p)}, g(p) \right).
\]
For every $p\in \Sigma$ and ${\bf v} \in T_p \Sigma$ we have
\[
d\Psi_p({\bf v})
=
\frac{1}{u(p)} \left({\bf v}(\psi_2), \ldots, {\bf v}(\psi_{n+1}),0\right) 
-
\frac{{\bf v}(u)}{u^2(p)} \left(\psi_2(p), \ldots,\psi_{n+1}(p),0\right) 
+
(0,\ldots,0, {\bf v}(g) ).
\]
Denote by $\langle\,,\rangle_0$ the standard metric of the Euclidean space $\mathbb{R}^{n+1}$. Hence, for every ${\bf v, w} \in T_p\Sigma$ we have ${\bf v}(g)=\displaystyle{ \frac{{\bf v}(w)}{u(p)}}$ and  
\begin{eqnarray*}
\langle d\Psi_p({\bf v}), d\Psi_p({\bf w})\rangle_0 &=& \frac{1}{u^2(p)} \sum_{i=2}^{n+1} {\bf v}(\psi_i) {\bf w}(\psi_i) + \frac{{\bf v}(u) {\bf w}(u)}{u^4(p)} \sum_{i=2}^{n+1} \psi^2_i(p) \\
& &- \frac{{\bf v}(u) }{u^3(p)} \sum_{i=2}^{n+1} {\bf w}(\psi_i) \psi_i - \frac{{\bf w}(u) }{u^3(p)} \sum_{i=2}^{n+1} {\bf v}(\psi_i) \psi_i + {\bf v} (g){\bf w} (g)\\
&=& \frac{1}{u^2(p)} \left( -{\bf v}(u){\bf w}(u) + \sum_{i=2}^{n+1} {\bf v}(\psi_i) {\bf w}(\psi_i)   +   { \bf v}(w) { \bf w}(w) \right) \\
&=&  \frac{1}{u^2(p)} \langle d\psi_p({\bf v}),  d\psi_p({\bf w}) \rangle_{\mathbb{L}^{n+2}}\\  &=&  \frac{1}{u^2(p)} \langle {\bf v},  {\bf w} \rangle_{\Sigma},
\end{eqnarray*}
where $\langle \,,\rangle_{\Sigma}$ denotes the Riemannian metric on $\Sigma$ induced by the immersion $\psi$.  In other words, 
\begin{equation}
\label{iso1}
\Psi^*(\langle\,,\rangle_0) = \frac{1}{u^2} \langle\,,\rangle_\Sigma .
\end{equation} 
It follows that $\Psi$ is a local diffeomorphism.
Assume now that $\Sigma$ is complete and $u$ satisfies \eqref{hyp:1}. By Lemma \ref{lem:complete} applied to the function $\lambda=u^{-(n-2)/2}$, we obtain that the conformal metric 
\[
\widetilde{\langle\,,\rangle} = \frac{1}{u^2}\langle\,,\rangle_\Sigma
\]
is also complete on $\Sigma$. Then, equation \eqref{iso1} means that the map
\[
\Psi:(\Sigma^n, \widetilde{\langle\,,\rangle}) \to (\Psi(\Sigma), \langle\,,\rangle_0)
\]
is a local isometry from a (connected)  complete  Riemannian manifold $\Sigma$ onto another connected Riemannian manifold $\Psi(\Sigma)$ of the same dimension.  By Lemma \ref{lem:coveringmap}, we obtain that $\Psi$ is a covering map and $\Psi(\Sigma)$ is complete. In particular, it implies that  $$\Psi(\Sigma)= \mathbb{S}^{n-1}\times \mathbb{R}.$$ 

Finally, since $\mathbb{S}^{n-1}\times \mathbb{R}$ is simply connected, we conclude that $\Psi$ is, in fact, a global diffeomorphism between $\Sigma$ and the cylinder  $\mathbb{S}^{n-1}\times \mathbb{R}$. 
\end{proof}

In the following example, we show that  the upper bound given by  \eqref{hyp:1} on the growth of $u$ is sharp.

\begin{example} Let  $f: \mathbb{R} \to (0,+\infty)$ be a smooth function.  We consider the warped product $\Sigma^n =\mathbb{R} \times_f \mathbb{S}^{n-1}$, where recall that  the warped product metric is given by 
\[ 
\langle  \, ,  \rangle_\Sigma=dt^2+f^2(t) \langle  \, , \rangle_{\mathbb{S}^{n-1}}.
\]
It is well known that a warped product of complete manifolds is complete for every warping function $f$ (see for instance  \cite[Lemma 7.40]{ONeill}).  Taken this into account, we have that $\Sigma^n$ is a complete manifold. We construct an immersion $\psi : \Sigma^n \to \mathbb{L}^{n+2}$ by setting
\[
\psi(p) = \psi(t,q)= (f(t) , f(t) q, t)
\]
for every $p=(t, q)\in \Sigma^n$.
Observe that
\[
\langle \psi(p),\psi(p) \rangle = -f^2(t)+ f^2(t)  \langle  q,  q \rangle_{\mathbb{S}^{n-1}} + t^2 = t^2 
\] 
and $u(p) = -\langle \psi(p), {\bf e}_1\rangle = f(t)  >0$.  Hence, $\psi(\Sigma)$ is contained in $\textcolor{black}{\Lambda^+ \times \mathbb{R}} \subset \mathbb{L}^{n+2}$. 

On the other hand, given $p = (t,q)\in \Sigma^n$ and ${\bf v} = (s, {\bf w})\in T_p \Sigma$, we obtain
\[ 
d\psi_p({\bf v})= (f'(t), f'(t) 	q + f(t) {\bf w}, s).
\]
Moreover,
\[
\langle d\psi_{p}({\bf v}),d\psi_{p}({\bf v})\rangle_{\mathbb{L}^{n+2}}
= -f'(t)^2+ f'(t)^2    + f^2(t)  \langle   {\bf w},   {\bf w} \rangle_{\mathbb{S}^{n-1}} + s^2
= \langle {\bf v},{\bf v}\rangle_{\Sigma}.
\]
It then follows that $\psi$ determines a spacelike isometric immersion from $\Sigma^n$ into the cylinder over the light cone $\textcolor{black}{\Lambda^+ \times \mathbb{R}}$.
By Theorem \ref{thm:01}, we obtain that 
$\Sigma$ is conformally diffeomorphic to the cylinder $\mathbb{S}^{n-1} \times \mathbb{R}$ if
\[
f(t) \leq C r(t,q) \log r(t,q), \qquad  r(t,q)\gg1.
\]

Now we will show that the upper bound \eqref{hyp:1} in Theorem \ref{thm:01} is sharp.  Let $f(t)=t^2+1$ be the warping function. Then  the function  $g(t) =\arctan t$  satisfies $dg = \frac{dt}{f(t)}$. In this case,  it is direct to see that \[\Psi(p)=\Psi(t,q) = (q, g(t)),\qquad \text{with $p\in \Sigma$},\] is an isometry from $(\Sigma, \widetilde{\langle\,,\rangle}) $ to the open set   $ \mathbb{S}^{n-1}\times (\pi/2,-\pi/{2}) \subset  \mathbb{S}^{n-1}\times  \mathbb{R}$, where
$\widetilde{\langle\,,\rangle} = \frac{1}{f^2}\langle\,,\rangle_\Sigma
 $.   In other words the manifold $\Sigma^n$, with the warped product metric, is conformally diffeomorphic to  $ \mathbb{S}^{n-1}\times (\pi/2,-\pi/{2})$.
In particular,  the metric $\frac{1}{f^2}{\langle\,,\rangle}_{\Sigma}$  is not complete. 
\end{example}

\begin{remark}
Every warped product  $\Sigma^n =\mathbb{R} \times_f \mathbb{S}^{n-1}$ is  conformally diffeomorphic to the cylinder $\mathbb{S}^{n-1} \times \mathbb{R}$ if $f$ is bounded from above. In this case, the map $\Phi$ is given by
\[
(t,q) \in \Sigma^n \mapsto (q, g(t)) \in \mathbb{S}^{n-1} \times \mathbb{R}
\] 
where $g(t) = \int_{t_0}^t \frac{d\tau}{f(\tau)}$ and  $\frac{1}{f^2}$ is the conformal factor.
\end{remark}

\textcolor{black}{Finally, under the existence of a non-vanishing spacelike coordinate function in the factor $\Lambda^+\subset\mathbb{L}^{n+1}$, we can assume without loss of generality that $\psi_{n+1}=\langle \psi, {\bf e}_{n+1} \rangle >0$, by further applying an isometry if necessary. Thus, following the same lines as in the proof of Theorem \ref{thm:02} and Theorem \ref{thm:01} we have the next result.}
\begin{theorem}
Let $\psi:  \Sigma^n \to\textcolor{black}{\Lambda^+ \times \mathbb{R}\subset} \mathbb{L}^{n+2}$ be a codimension two complete spacelike submanifold that factors through $ \Lambda^+ \times \mathbb{R}$, \textcolor{black}{with $n\geq 2$}. Assume that the function $v=\psi_{n+1}=\langle \psi, {\bf e}_{n+1} \rangle $ satisfies 
\begin{equation}
\textcolor{black}{0<v(p) \leq C r(p) \log(r(p))}, \qquad r(p)\gg1,
\end{equation}
where $C$ is a positive constant and $r$ denotes the Riemannian distance function from a fixed origin $o \in \Sigma$. Let $w = \langle \psi, {\bf e}_{n+2}\rangle$.
If $\displaystyle{ \frac{dw}{v}}$ is an exact form  on $\Sigma$, then 
$\Sigma$ is conformally diffeomorphic to the product $\mathbb{H}^{n-1} \times \mathbb{R}.$  
\end{theorem}
\textcolor{black}{For its proof, replace the map $\Psi:\Sigma^n\to\mathbb{S}^{n-1}\times\mathbb{R}$ in the proof of Theorem \ref{thm:01} by the new map $\Psi:\Sigma^n\to\mathbb{H}^{n-1}\times\mathbb{R}$ given now by
\[
\Psi(p)=\left(\frac{\psi_1(p)}{v},\ldots,\frac{\psi_n(p)}{v},g(p)\right),
\]
and observe that
\[
\Psi^*(\langle\,,\rangle_0) = \frac{1}{v^2} \langle\,,\rangle_\Sigma,
\] 
where $\langle\,,\rangle_0$ stands here for the metric of $\mathbb{H}^{n-1}\times\mathbb{R}$.
} 

\section{Nullcones in the De Sitter spacetime}

\textcolor{black}{As is well-known}, the ($n+2$)-dimensional de Sitter spacetime $\mathbb{S}^{n+2}_1$ is the complete and simply connected manifold with constant sectional curvature $K=1$. It can be represented as the hyperquadric
\[
\mathbb{S}^{n+2}_1 = \{ x \in \mathbb{L}^{n+3} \,\, |\,\, \langle x,x\rangle =1\},
\]
\textcolor{black}{with $x=(x_1, x_2, \ldots, x_{n+3})\in\mathbb{L}^{n+3}$ and
$\langle ,\rangle =-dx_1^2+dx_2^2+\cdots+dx_{n+3}^2$.}

Let ${\bf a} \in \mathbb{S}_1^{n+2}$ be a fixed point of de Sitter spacetime. The {\em future light cone} in $\mathbb{S}_1^{n+2}$  with vertex at ${\bf a}$ is the subset
 \[
\Lambda^+_{\bf a}= \{x \in \mathbb{S}_1^{n+2} \,\,|\,\, \langle {\bf a}, x \rangle=1, \, \textcolor{black}{\langle x-{\bf a}, {\bf e}_1\rangle =-x_1+a_1 <0} \}.
\]
Equivalently, $x \in \mathbb{S}_1^{n+2}$ belongs to $\Lambda^+_{\bf a}$ if and only if $x-{\bf a}$ is a future null vector. By applying an isometry, we may assume that the vertex of the light cone is the point \textcolor{black}{${\bf a}=(0,\ldots, 0,1)\in \mathbb{S}_1^{n+2}$}. 

Let $\psi : \Sigma \to \Lambda^+\subset\mathbb{S}^{n+2}_1$ be a codimension two spacelike submanifold which is contained in the future component of the light cone 
\[
\Lambda^+=\Lambda^+_{\bf a} =\{x \in \mathbb{S}_1^{n+2}\, \,|\, \,x_{n+3}=1, \, x_1>0\}.
\] 
The tools developed in Section \ref{sec:LC} enable us to prove  the following results mutatis mutandis. Therefore we do not include the proofs. We remark that our results rely on the analysis of the radial function \textcolor{black}{$w=\psi_{n+2}=\langle \psi, {\bf e}_{n+2} \rangle$}. For related results pertaining a height function, see \cite{ACR2}.

\begin{theorem} 
Let $\psi: \Sigma^n \to \Lambda^+ \subset \mathbb{S}_1^{n+2}$ be a codimension two  complete spacelike submanifold that factors through $ \Lambda^+$. Assume that the  function \textcolor{black}{$w=\psi_{n+2}=\langle \psi, {\bf e}_{n+2} \rangle$} satisfies 
\begin{equation} \label{hype:deSitter}
0<w(p) \leq C r(p) \log(r(p)), \qquad r(p)\gg1,
\end{equation}
where $C$ is a positive constant and $r$ denotes the Riemannian distance function from a fixed origin $o \in \Sigma$. Then 
$\Sigma^n$ is conformally diffeomorphic to the hyperbolic space $\mathbb{H}^{n}$.
\end{theorem}

\begin{example} Given a positive differentiable function $f:\mathbb{H}^n\to (0,\infty)$ we can construct an embedding 
$\psi_f :\mathbb{H}^n\to  \Lambda^+ \subset \mathbb{S}_1^{n+2}$ by setting
\[
\psi_f(p)=( f(p) p , f(p), 1).
\] 
\textcolor{black}{Note that $\langle\psi_f(p),\psi_f(p)\rangle=f^2(p)\langle p,p\rangle_{\mathbb{L}^{n+1}}+f^2(p)+1=1$, 
$(\psi_f)_{n+3}=1$ and $(\psi_f)_1=f(p)p_1>0$, so that $\psi_f(\mathbb{H}^n)\subset\Lambda^+$.}
\textcolor{black}{
For every $p\in\mathbb{H}^n$ and ${\bf v}\in T_p\mathbb{H}^n$, we have
\[
d(\psi_f)_p({\bf v})=({\bf v}(f)p+f(p){\bf v},{\bf v}(f), 0),
\]
from which one gets
\[
\langle 
d(\psi_f)_p({\bf v}),  d(\psi_f)_p({\bf w})
\rangle_{\mathbb{S}_1^{n+2}}
= 
f^2(p)
\langle 
{\bf v}, {\bf w}
\rangle_{\mathbb{H}^n}.
\]
Therefore, $\psi_f$ defines a spacelike immersion of $\mathbb{H}^n$ into $\Lambda^+\subset\mathbb{S}^{n+2}_1$ whose induced metric is conformal to the standard hyperbolic metric.}
\end{example}

\begin{corollary}
Let $\psi:  \Sigma^n \to \Lambda^+ \subset \mathbb{S}^{n+2}_1$ be a codimension two complete spacelike submanifold that factors through the nullcone $\Lambda^+$. Assume that $w= \langle \psi, {\bf e}_{n+2} \rangle$ satisfies the bounds \eqref{hype:deSitter}. Then \textcolor{black}{there exists a conformal} diffeomorphism $\Psi: (\Sigma^n, \langle \,,\rangle_\Sigma) \to  (\mathbb{H}^n, \langle \,,\rangle_{\mathbb{H}^n} ) $ such that 
\[
\psi = \psi_f\circ \Psi 
 \]
 where  $f=\psi_{n+2}\circ \Psi^{-1}$  and $\psi_f: \mathbb{H}^n \to \Lambda^+ \subset \mathbb{S}^{n+1}_1$ is the embedding 
 \[
\psi_f(p)= (f(p) p , f(p), 1).
\]
In particular, the immersion $\psi$ is an embedding.
\[
\xymatrix{
\Sigma^n  \ar[r]^-{w}  & (0,+\infty)\\
 \mathbb{H}^n \ar[ru]_-{f}   \ar[u]^-{\Psi^{-1}} 
} 
\qquad  
\qquad 
\textcolor{black}{\xymatrix{
\Sigma^n  \ar[r]^-{\psi} \ar[d]_-{\Psi} & \Lambda^{+}\subset\mathbb{S}^{n+2}_1\\
 \mathbb{H}^n \ar[ru]_-{\psi_f}   
}}
\]
\end{corollary}

In \cite[Section 3]{NPS2016},  the authors found a \textcolor{black}{general} method to construct null hypersurfaces in Lorentzian warped products. In particular they gave several examples of null hypersurfaces in both de Sitter and anti de Sitter spaces. One of them was \textcolor{black}{\cite[Example 4.2]{NPS2016}}; for the sake of completeness we give some details of it.

First we recall that the de Sitter space 
$\mathbb{S}_1^{n+2}\subset \mathbb{L}^{n+3}$ 
also may be described as the warped product
$ -\mathbb{R}\times_{\cosh}\mathbb{S}^{n+1}$ by using the isometry  $\mathfrak{F}:-\mathbb{R}\times_{\cosh} \mathbb{S}^{n+1}  \to \mathbb{S}_1^{n+2} $ given by 
\begin{equation}\label{eq:iso:deSitter}
\mathfrak{F}(t,q)  =(\sinh t, (\cosh t)  q) =x \in \mathbb{S}_1^{n+2}
\end{equation}
for every $(t,q) \in -\mathbb{R}\times_{\cosh} \mathbb{S}^{n+1}$.

We now consider the hypersurface $S\subset \mathbb{S}^{n+1}$ defined by the set of points making an angle $\theta_0$ with the fixed vector \textcolor{black}{${\bf e}_{n+3}=(0,0,\ldots,0,1) \in \mathbb{S}^{n+1}$}.  Hence
\[
S=\{ q \in \mathbb{S}^{n+1}\mid \langle {\bf e}_{n+3} , q\rangle = \cos \theta_0\}
\]
and $S$ can be considered as a parallel in $\mathbb{S}^{n+1}$, see Figure \ref{fig1}. 

\begin{figure}[ht]
\centering
\begin{tikzpicture}[scale=0.67]
\shade[ball color=blue!80,opacity=0.2] (0,0) circle (3.5cm);
\draw (0,0) node at (5.4,1.5){} circle (3.5cm);
\draw (-3.5,0) arc (180:360:3.5 and 0.6);
\draw[dashed] (3.5,0) arc (0:180:3.5 and 0.6);
\draw[-] (0,-4 ) --  (0,5) node[right]{$x_{n+3}$};  
\fill[fill=red] (3,1.86) circle (2pt);  
\draw[-,thick] (0,0) -- (3,1.86) node[right]{$\,q\in S\subset \mathbb{S}^{n+1}$};
\draw[->,thick] (0,0) -- (0,3.5) node[below]{$\,\,\qquad{\bf e}_{n+3}$};
\fill[fill=black] (0,0) circle (1.5pt);
\draw[red,thick] (-3,1.86) arc (180:360:3 and .3);
\draw[red,thick,dashed] (-3,1.86) arc (180:0:3 and .3);
 \draw[black,thick] (0.8,0.4) arc (46:95:1.1)
node[above]{$\,\,\,\qquad\theta_0$};
\end{tikzpicture}
\caption{}
\label{fig1}
\end{figure}

Given $q_0\in S$,   by \textcolor{black}{\cite[Proposition 3.3]{NPS2016}}, there is a neighborhood $U$ of $q_0$ in $\mathbb{S}^{n+1}$ and a function $f: U \to \mathbb{R}$ such that the graph of $f$ given by  
\[
M=\{(f(q),q) \,|\, q\in U\}
\]
 is   a null hypersurface in $-\mathbb{R}\times_{\cosh} \mathbb{S}^{n+1}$.  In this case we have  
\[
f(q) =( g^{-1} \circ r)(q) 
\]
where $r:\mathbb{S}^{n+1}\to \mathbb{R}$ is called the signed distance function  to the parallel $S$ given by 
\[
r(q) = \theta_0 - \cos^{-1}\langle {\bf e}_{n+3} , q\rangle 
\qquad \text{and} 
\qquad
g(s)= \int_{0}^{s} \frac{d\tau}{ \cosh (\tau)}, \qquad s\in \mathbb{R}.
\]
Then the map $f$ becomes 
\[
f(q) = \sinh^{-1} \circ \tan (\theta_0-\cos^{-1}\langle {\bf e}_{n+3} , q\rangle )
\]

Using the isometry \eqref{eq:iso:deSitter}, straightforward computations imply that $M$ can represented as the intersection of a hyperplane with $\mathbb{S}^{n+2}_1 \subset \mathbb{L}^{n+3}$. In fact, 
\[
\mathfrak{F}(M)= \{ (x_1, x_2 , \ldots,x_{n+3}) \in \mathbb{S}^{n+2}_1 \mid x_{n+3}= \alpha + \sqrt{ 1- \alpha^2} x_1 \},
\]
where $\cos \theta_0 = \alpha.$  Note that $\mathfrak{F}(M)$ can be parametrized by
\[
(s, u_1, \ldots , u_n) \mapsto ( s , R(s) \varphi(u_1, \ldots, u_n) , \alpha + \sqrt{1-\alpha^2 }s )
\]
where $R(s)= \alpha s - \sqrt{1-\alpha^2}$ and $\varphi$ is an orthogonal parametrization of $\mathbb{S}^{n}$. Here 
\[
\textcolor{black}{\frac{\partial}{\partial s}}= (1, \alpha \varphi(u_1, \ldots, u_n), \sqrt{1-\alpha^2} ) 
\]
is a null vector field tangent to $\mathfrak{F}(M)$.
\begin{definition} Let $0\leq \alpha \leq 1$. The {\em nullcone} $\Lambda_\alpha^+\subset \mathbb{S}_1^{n+2}$  is the set 
\[
\Lambda_\alpha^+ = \{ (x_1, x_2 , \ldots,x_{n+3}) \in \mathbb{S}^{n+2}_1 
\mid x_{n+3}= \alpha + \sqrt{ 1- \alpha^2} x_1, \  x_1>0 \}.
\] 
\end{definition}
 Observe that the light cone with the vertex in ${\bf e}_{n+3}=(0,\ldots, 0,1)\in \mathbb{S}_1^{n+2}$ corresponds to the case $\alpha=1$.  That is 
 \[
 \Lambda^+=\Lambda^+_1=\{ 
 x \in \mathbb{S}^{n+2}_1 
\mid x_{n+3}=1, \  x_1>0 \}.
\]
On the other hand, the case $\alpha=0$ corresponds to
\[
\Lambda^+_0=\{ 
 x \in \mathbb{S}^{n+2}_1 
\mid x_{n+3}=x_1, \  x_1>0 \}.
\]

\begin{theorem}\label{thm:GCone}
Let $\psi:\Sigma^n \to \Lambda_\alpha^+ \subset \mathbb{S}_1^{n+2}$ be a codimension two spacelike submanifold that factors through $\Lambda^+_\alpha$. In the case where $0<\alpha<1$, assume further that $\psi(\Sigma)$ is contained either in $\displaystyle{\Lambda_{\alpha,-}^+ = \{ x\in \Lambda_{\alpha}^+\mid 0<x_1<\frac{\sqrt{1-\alpha^2}}{\alpha} \}}$ or in $\displaystyle{\Lambda_{\alpha,+}^+ = \{ x\in \Lambda_{\alpha}^+\mid x_1>\frac{\sqrt{1-\alpha^2}}{\alpha} \}}.$
Then  $\Sigma^n$ is locally diffeomorphic to $\mathbb{S}^{n}$.
\end{theorem}
\begin{proof}For every $p\in \Sigma$, $\psi(p)=(\psi_1(p), \psi_{2}(p), \ldots,  \psi_{n+3}(p)) \in \Lambda_\alpha^+$. Thus
$\psi_{n+3} =\alpha + \sqrt{ 1- \alpha^2}  \psi_1$ and
\[
-\psi_1^2(p) + \psi_2^2(p)+\cdots+\psi_{n+2}^2(p) + \psi_{n+3}^2(p)=1.
\] 
It follows that 
\[
\sum_{i=2}^{n+2} \psi_i^2 = R(\psi_1)^2>0
\qquad
\text{where} 
\quad 
R(\psi_1) = \alpha\,\psi_1  - \sqrt{ 1- \alpha^2}.
\]
\textcolor{black}{Observe that when $\alpha=0$, $R(\psi_1)=-1<0$ and when $\alpha=1$, $R(\psi_1)=\psi_1>0$ on $\Sigma$. On the other hand, when $0<\alpha<1$ it follows that $R(\psi_1)<0$ on $\Sigma$ if $\psi(\Sigma)\subset\Lambda^+_{\alpha,-}$, while $R(\psi_1)>0$ on $\Sigma$ if $\psi(\Sigma)\subset\Lambda^+_{\alpha,+}$.} 
\textcolor{black}{Therefore, in any case $R(\psi_1)^2>0$ on $\Sigma$ and we can define the map} $\Psi:\Sigma^n\to\mathbb{S}^{n}$ by
\[
\Psi(p)= \left( \frac{\psi_2(p)}{R(\psi_1)}, \ldots,   \frac{\psi_{n+2}(p)}{R(\psi_1)} \right).
\]

For every $p\in \Sigma$ and ${\bf v} \in T_p \Sigma$ we have
\[
d\Psi_p({\bf v})
=
\frac{1}{R(\psi_1)} \left({\bf v}(\psi_2), \ldots, {\bf v}(\psi_{n+2})\right) 
-
\frac{\alpha {\bf v}(\psi_1)}{R(\psi_1( p))^2} \left(\psi_2(p), \ldots,\psi_{n+2}(p)\right).
\]

Denote by $\langle\,,\rangle_{\mathbb{S}^n}$ the standard round metric of the sphere $\mathbb{S}^{n}$, then  for every ${\bf v, w} \in T_p\Sigma$ we obtain
\begin{align*}
\langle d\Psi_p({\bf v}), d\Psi_p({\bf w})\rangle_{\mathbb{S}^n} =
&\frac{1}{R(\psi_1)^2}\sum_{i=2}^{n+2}{\bf v}(\psi_i) {\bf w}(\psi_i)  
+\frac{\alpha^2 {\bf v}(\psi_1) {\bf w}(\psi_1)}{R(\psi_1)^4}\sum_{i=2}^{n+2}\psi^2_i(p)\\
& - \frac{\alpha}{R(\psi_1)^3}\left({\bf w}(\psi_1)\sum_{i=2}^{n+2} \psi_i{\bf v}(\psi_i)+{\bf v}(\psi_1)\sum_{i=2}^{n+2} \psi_i {\bf w}(\psi_i)\right) \\
&= 
\frac{1}{R(\psi_1)^2}
\left(
-\alpha^2  {\bf v}(\psi_1) {\bf w}(\psi_1)
+ 
\sum_{i=2}^{n+2}{\bf v}(\psi_i) {\bf w}(\psi_i)  
\right),
\end{align*}
since
\begin{eqnarray*}
{\bf w}(\psi_1)\sum_{i=2}^{n+2} \psi_i{\bf v}(\psi_i)+{\bf v}(\psi_1)\sum_{i=2}^{n+2} \psi_i {\bf w}(\psi_i) & = &
{\bf w}(\psi_1)\frac{1}{2}{\bf v}(\sum_{i=2}^{n+2}\psi^2_i)+{\bf v}(\psi_1)\frac{1}{2}{\bf w}(\sum_{i=2}^{n+2}\psi^2_i)\\
{} & = & {\bf w}(\psi_1)\frac{1}{2}{\bf v}(R(\psi_1)^2)+{\bf v}(\psi_1)\frac{1}{2}{\bf w}(R(\psi_1)^2)\\
{} & = & \alpha R(\psi_1){\bf v}(\psi_1){\bf w}(\psi_1).
\end{eqnarray*}
Then, using that
\[
{\bf v}(\psi_{n+3}){\bf w}(\psi_{n+3})={\bf v}(\psi_{1}){\bf w}(\psi_{1})-\alpha^2{\bf v}(\psi_{1}){\bf w}(\psi_{1}),
\]
we conclude that
\begin{align*}
\langle d\Psi_p({\bf v}), d\Psi_p({\bf w})\rangle_{\mathbb{S}^n}=& 
\frac{1}{R(\psi_1)^2}
\left( -{\bf v}(\psi_1){\bf w}(\psi_1) + \sum_{i=2}^{n+2} {\bf v}(\psi_i) {\bf w}(\psi_i)   +   { \bf v}(\psi_{n+3}) { \bf w}(\psi_{n+3}) \right) \\
=&  \frac{1}{R(\psi_1)^2} \langle d\psi_p({\bf v}),  d\psi_p({\bf w}) \rangle_{\mathbb{S}_1^{n+2}}  =  
\frac{1}{R(\psi_1)^2} \langle {\bf v},  {\bf w} \rangle_{\Sigma}.
\end{align*}
Hence 
\begin{equation}
\label{luisA1}
\Psi^*(\langle \,,\rangle_{\mathbb{S}^n}) = 
\frac{1}{R(\psi_1)^2} \langle \,,  \rangle_{\Sigma},
\end{equation}
which implies that $\Psi$ is a local diffeomorphism.   
\end{proof}

\textcolor{black}{In particular, and as a direct application of Theorem \ref{thm:GCone}, we obtain the following consequence, which extends Proposition 4.1 in \cite{ACR2}. Observe that Proposition 4.1 in \cite{ACR2} corresponds to the case $\alpha=1$, while our proposition below corresponds to the rest of the cases $0\leq \alpha<1$.}
\begin{proposition}
\label{prop:luisA6}
Let $\psi:  \Sigma^n \to \Lambda_\alpha^+ \subset \mathbb{S}_1^{n+2}$ be a codimension two  spacelike  submanifold that factors through $ \Lambda_\alpha^+$. \textcolor{black}{In the case where $0<\alpha<1$, assume further that $\psi(\Sigma)$ is contained either in $\displaystyle{\Lambda_{\alpha,-}^+ = \{ x\in \Lambda_{\alpha}^+\mid 0<x_1<\frac{\sqrt{1-\alpha^2}}{\alpha} \}}$ or in $\displaystyle{\Lambda_{\alpha,+}^+ = \{ x\in \Lambda_{\alpha}^+\mid x_1>\frac{\sqrt{1-\alpha^2}}{\alpha} \}}.$ Assume that $\Sigma$ is complete.
\begin{enumerate}
\item[(a)] If $\alpha=0$, then $\Sigma$ is compact and isometric to the round sphere $\mathbb{S}^n$.
\item[(b)] If $0<\alpha<1$ and $\psi(\Sigma)\subset\Lambda^+_{\alpha,-}$, then $\Sigma$ is compact and conformally diffeomorphic to the round sphere $\mathbb{S}^{n}$.
\item[(c)] If $0<\alpha<1$ and $\psi(\Sigma)\subset\Lambda^+_{\alpha,+}$,
and the  positive function $u=\psi_1=-\langle \psi, {\bf e}_{1} \rangle$ satisfies 
\begin{equation}
\label{luisA2}
u(p) \leq C r(p) \log(r(p)), \qquad r(p)\gg1,
\end{equation} 
where $C$ is a positive constant and $r$ denotes the Riemannian distance function from a fixed origin $o \in \Sigma$, then 
$\Sigma$ is compact and conformally diffeomorphic to the round sphere $\mathbb{S}^{n}$. In particular, this holds if $\sup_\Sigma u < +\infty$ and, more generally, if $\limsup_{r\to \infty} \frac{u}{r \log r} < +\infty$.
\end{enumerate}}
\end{proposition}
\begin{proof} 
From the proof of Theorem \ref{thm:GCone}, we already know from \eqref{luisA1} that the map $\Psi:(\Sigma^n,\widetilde{\langle,\rangle})\to (\mathbb{S}^n,\langle,\rangle_{\mathbb{S}^n})$ is a local isometry, where 
\[
\widetilde{\langle,\rangle}=\frac{1}{R(u)^2}\langle,\rangle_\Sigma \quad \text{, with } u=\psi_1.
\]
Recall also that the metric $\langle,\rangle_\Sigma$ is complete on $\Sigma$, by hypothesis. 

In the case $\alpha=0$, $R(u)=-1$. Therefore, $\widetilde{\langle,\rangle}=\langle,\rangle_\Sigma$ is a complete metric on $\Sigma$ and the map 
$\Psi:(\Sigma^n,\langle,\rangle_\Sigma)\to (\mathbb{S}^n,\langle,\rangle_{\mathbb{S}^n})$
is a local isometry which, by Lemma \ref{lem:coveringmap}, is in fact a covering map and hence a global isometry.

In the case where $0<\alpha<1$ and $\psi(\Sigma)\subset\Lambda^+_{\alpha,-}$ we have $\displaystyle{0<u<\frac{\sqrt{1-\alpha^2}}{\alpha}}$ and hence
$\displaystyle{-\sqrt{1-\alpha^2}<R(u)<0}$. In this case, the conformal metric $\widetilde{\langle,\rangle}$ on $\Sigma$ satisfies 
\[
\widetilde{\langle,\rangle}=\frac{1}{R(u)^2}\langle,\rangle_\Sigma>
\frac{1}{1-\alpha^2}\langle,\rangle_\Sigma,
\]
which implies that $\widetilde{\langle,\rangle}$ is also a complete metric on $\Sigma$. Then, in this case the map $\Psi:(\Sigma^n,\langle,\rangle_\Sigma)\to (\mathbb{S}^n,\langle,\rangle_{\mathbb{S}^n})$
is also a local isometry which, by Lemma \ref{lem:coveringmap}, is in fact a covering map and hence a global isometry.

Finally, in the case $0<\alpha<1$ and $\psi(\Sigma)\subset\Lambda^+_{\alpha,+}$ we have $\displaystyle{u>\frac{\sqrt{1-\alpha^2}}{\alpha}}$ and hence
$\displaystyle{R(u)=\alpha u-\sqrt{1-\alpha^2}>0}$.  Therefore, by \eqref{luisA2} the positive function $R(u)$ also satisfies
\[
R(u)(p)\leq C r(p) \log(r(p)), \qquad r(p)\gg1,
\]
and by Lemma \ref{lem:complete} the conformal metric $\widetilde{\langle,\rangle}$ is also complete on $\Sigma^n$, and using again Lemma \ref{lem:coveringmap} we conclude that $\Psi$ is a global diffeomorphism, as in the proof of our previous Theorem \ref{thm:02} and Theorem \ref{thm:01}. 
\end{proof}

\begin{corollary}
Let $\psi:  \Sigma^n \to \Lambda_\alpha^+ \subset \mathbb{S}_1^{n+2}$ be a codimension two compact spacelike  submanifold that factors through $ \Lambda_\alpha^+$. \textcolor{black}{In the case where $0<\alpha<1$, assume further that $\psi(\Sigma)$ is contained either in $\displaystyle{\Lambda_{\alpha,-}^+ = \{ x\in \Lambda_{\alpha}^+\mid 0<x_1<\frac{\sqrt{1-\alpha^2}}{\alpha} \}}$ or in $\displaystyle{\Lambda_{\alpha,+}^+ = \{ x\in \Lambda_{\alpha}^+\mid x_1>\frac{\sqrt{1-\alpha^2}}{\alpha} \}}.$} Then 
$\Sigma^n$ is conformally diffeomorphic to the round sphere $\mathbb{S}^{n}$.
\end{corollary}

\begin{example}
\label{ex:luisA3}
Given a positive smooth function $f: \mathbb{S}^{n}\to (0,\infty)$ we can construct an embedding $\psi_{f,\alpha} : \mathbb{S}^n\to \Lambda_\alpha^+ \subset \mathbb{S}_1^{n+2}$, with $\alpha=0$ or $\alpha=1$, by setting 
\[
\psi_{f, \alpha} (p) = 
(f(p), R(f(p)) p , \alpha + \sqrt{1-\alpha^2} f(p)) ,
\]
where $R(s)=\alpha s-\sqrt{ 1- \alpha^2}$. In other words,
\[
\psi_{f,0}(p)=(f(p),-p,f(p))\in\Lambda^+_0 \quad \text{and} \quad 
\psi_{f,1}(p)=(f(p),f(p)p,1)\in\Lambda^+_1.
\]
The case $\alpha=1$ corresponds to \cite[Example 4.1]{ACR2} and it satisfies 
\[
\langle,\rangle=\psi_{f,1}^* (\langle \,,\rangle) = f^2 \langle\,, \rangle_{\mathbb{S}^n},
\]
which means that $\psi_{f,1} $ determines a spacelike immersion of $\mathbb{S}^n$ into $\Lambda_{1}^+$ whose induced metric is conformal to the standard metric of the round sphere with conformal factor $f$.

The case $\alpha=0$ is new. In that case, For each  ${\bf v}\in T_p \mathbb{S}^n$ we have
\[
d(\psi_{f,0})_p ({\bf v}) = 
({\bf v}(f), -{\bf v}, {\bf v}(f))
\]
and 
\[
\langle
d(\psi_{f,0})_p ({\bf v}),d(\psi_{f,0})_p ({\bf w})
\rangle=\langle {\bf v}, {\bf w}\rangle_{\mathbb{S}^n}.
\]
Hence
\[
\langle,\rangle=\psi_{f,0}^*(\langle \,,\rangle)=\langle\,, \rangle_{\mathbb{S}^n},
\]
which means that $\psi_{f,0} $ determines a spacelike isometric immersion of the round sphere $\mathbb{S}^n$ into $\Lambda_{0}^+$. 
\end{example}

\begin{example}
\label{ex:luisA4}
Let $0<\alpha<1$. For a positive smooth function $$\displaystyle{f:\mathbb{S}^{n}\to \left(0,\frac{\sqrt{1-\alpha^2}}{\alpha}\right)}$$ we can construct an embedding $\psi_{f,\alpha} : \mathbb{S}^n\to \Lambda_{\alpha,-}^+ \subset \mathbb{S}_1^{n+2}$ by setting 
\[
\psi_{f, \alpha} (p) = 
(f(p), R(f(p)) p , \alpha + \sqrt{1-\alpha^2} f(p)) ,
\]
where $R(f)=\alpha f-\sqrt{ 1- \alpha^2}<0$. Observe that $\langle \psi_{f, \alpha} (p) , \psi_{f, \alpha} (p)   \rangle=1$ for every $p\in \mathbb{S}^n$. For each  ${\bf v} \in T_p \mathbb{S}^n$ we have
\[
d(\psi_{f, \alpha})_p ({\bf v}) = 
({\bf v}(f) , \alpha {\bf v}(f) p +R(f(p)) {\bf v},  \sqrt{1-\alpha^2} {\bf v}(f))
\]
and 
\[
\langle
d(\psi_{f, \alpha})_p ({\bf v}),d(\psi_{f, \alpha})_p ({\bf w}\rangle
= R(f(p))^2 \langle {\bf v}, {\bf w}\rangle_{\mathbb{S}^n }.
\]
Hence
\[
\langle \,,\rangle=\psi_{f,\alpha}^* (\langle \,,\rangle) = 
(R\circ f)^2 \langle\,, \rangle_{\mathbb{S}^n}=
(\alpha f-\sqrt{ 1- \alpha^2})^2 \langle\,, \rangle_{\mathbb{S}^n},
\]
which means that $\psi_{f,\alpha} $ determines a spacelike immersion of $\mathbb{S}^n$ into $\Lambda_{\alpha,-}^+$ whose induced metric is conformal to the standard metric of the round sphere, with conformal factor $-R(f)=\sqrt{ 1- \alpha^2}-\alpha f>0$.
\end{example}

\begin{example}
\label{ex:luisA5}
Let $0<\alpha<1$. Similarly as in Example \ref{ex:luisA4}, for a positive smooth function $$\displaystyle{f:\mathbb{S}^{n}\to \left(\frac{\sqrt{1-\alpha^2}}{\alpha},+\infty\right)}$$ we can construct an embedding $\psi_{f,\alpha} : \mathbb{S}^n\to \Lambda_{\alpha,+}^+ \subset \mathbb{S}_1^{n+2}$ by setting 
\[
\psi_{f, \alpha} (p) = 
(f(p), R(f(p)) p , \alpha + \sqrt{1-\alpha^2} f(p)) ,
\]
where $R(f)=\alpha f-\sqrt{ 1- \alpha^2}>0$. Computing as in Example \ref{ex:luisA4}, we can easily check that
\[
\langle \,,\rangle=\psi_{f,\alpha}^* (\langle \,,\rangle) = 
(R\circ f)^2 \langle\,, \rangle_{\mathbb{S}^n}=
(\alpha f-\sqrt{ 1- \alpha^2})^2 \langle\,, \rangle_{\mathbb{S}^n},
\]
which means that $\psi_{f,\alpha} $ determines a spacelike immersion of $\mathbb{S}^n$ into $\Lambda_{\alpha,+}^+$ whose induced metric is conformal to the standard metric of the round sphere, with conformal factor $R(f)=\alpha f-\sqrt{ 1- \alpha^2}>0$.
\end{example}


In our next result, and as a direct consequence of Proposition \ref{prop:luisA6} we obtain that every codimension two compact spacelike submanifold contained in $\Lambda^+_0$ is, up to a global isometry, as in Example \ref{ex:luisA3}. When $\alpha=1$, the corresponding result for the case of submanifolds contained in $\Lambda^+_1$ is Corollary 4.1 in \cite{ACR2}.
\begin{theorem}  
\label{thm:luisA7}
Let $\psi:\Sigma^n \to \Lambda_0^+ \subset \mathbb{S}_1^{n+2}$ be a codimension two compact spacelike submanifold that factors through $\Lambda_0^+$.  Then there exists a global isometry 
$\Psi:(\Sigma^n, \langle \,,\rangle) \to  (\mathbb{S}^n, \langle \,,\rangle_{\mathbb{S}^n})$  such that
\[
\Psi^* (\langle \,, \rangle)=\langle \, ,\rangle_{\mathbb{S}^{n}},
\]
and $\psi = \psi_{f,0} \circ \Psi$  where  $f = u\circ \Psi^{-1} : \mathbb{S}^{n} \to (0,+\infty)$, with $u=\psi_1=-\langle \psi, {\bf e}_1 \rangle>0$, and $\psi_{f,0}: \mathbb{S}^n \to \Lambda_0^+ \subset \mathbb{S}_1^{n+2}$ is the embedding 
\[
\psi_{f,0} (p) = (f(p), -p , f(p)). 
\]
\[
\xymatrix{
\Sigma^n  
\ar[r]^-{u}
& (0,+\infty)
\\
\mathbb{S}^n 
\ar[ru]_-{f}   
\ar[u]^-{\Psi^{-1}} 
} 
\qquad  
\qquad 
\xymatrix{
\Sigma^n  
\ar[r]^-{\psi} 
\ar[d]_-{\Psi}  & \Lambda_0^+\subset \mathbb{S}_1^{n+2} \\
 \mathbb{S}^n \ar[ru]_-{\psi_{f,\alpha}}   
} 
\]
In particular, the immersion $\psi$ is an embedding.
\end{theorem}
\begin{proof}
We consider the map $\Psi$ as in the proof of Theorem \ref{thm:GCone} and Proposition \ref{prop:luisA6}, and recall that in this situation 
$\Psi:(\Sigma^n,\langle,\rangle_\Sigma)\to(\mathbb{S}^{n},\langle,\rangle_{\mathbb{S}^n})$ is a global isometry.
Let $f= u \circ \Psi^{-1}$. Then $f\circ \Psi = u$ and 
\begin{align*}
(\psi_{f,0} \circ \Psi)(p) &=
\left( f(\Psi(p)), -\Psi(p), f(\Psi(p)) \right)\\
&= (u(p), -\Psi(p), u(p))\\
&=( \psi_1(p), \psi_2(p), \ldots, \psi_{n+2}(p), \psi_1(p)) \\
&= \psi(p).
\end{align*}
\end{proof}


Finally, Proposition \ref{prop:luisA6} also provides us with the result corresponding to Theorem \ref{thm:luisA7} but for the remaining cases where $0<\alpha<1$.
\begin{theorem}  
Let $\psi:  \Sigma^n \to \Lambda_\alpha^+ \subset \mathbb{S}_1^{n+2}$ be a codimension two compact spacelike submanifold that factors through $\Lambda_\alpha^+$.  \textcolor{black}{Assume further that $\psi(\Sigma)$ is contained either in $\displaystyle{\Lambda_{\alpha,-}^+ = \{ x\in \Lambda_{\alpha}^+\mid 0<x_1<\frac{\sqrt{1-\alpha^2}}{\alpha} \}}$ or in $\displaystyle{\Lambda_{\alpha,+}^+ = \{ x\in \Lambda_{\alpha}^+\mid x_1>\frac{\sqrt{1-\alpha^2}}{\alpha} \}}$}. 
Then there exists a conformal diffeomorphism $\Psi: (\Sigma^n, \langle \,,\rangle) \to  (\mathbb{S}^n, \langle \,,\rangle_{\mathbb{S}^n }) $  such that
\[
\Psi^* (\langle \,, \rangle) = \frac{1}{R(u)^2} \langle \, ,\rangle_{\mathbb{S}^{n}},
\qquad
R(u) = \alpha \, u - \sqrt{ 1- \alpha^2},
\]
with $u=\psi_1=-\langle \psi, {\bf e}_1 \rangle$  and $\psi = \psi_{f,\alpha} \circ \Psi$  where  $f=u\circ\Psi^{-1} : \mathbb{S}^{n} \to (0,+\infty)$ and $\psi_{f,\alpha}: \mathbb{S}^n \to \Lambda_\alpha^+ \subset \mathbb{S}_1^{n+2}$ is the embedding 
\[
\psi_{f, \alpha} (p) = 
(f(p), R(f(p))\, p , \alpha + \sqrt{1-\alpha^2} f(p)). 
\]
\[
\xymatrix{
\Sigma^n  
\ar[r]^-{u}
& (0,+\infty)
\\
\mathbb{S}^n 
\ar[ru]_-{f}   
\ar[u]^-{\Psi^{-1}} 
} 
\qquad  
\qquad 
\xymatrix{
\Sigma^n  
\ar[r]^-{\psi} 
\ar[d]_-{\Psi}  & \Lambda_\alpha^+\subset \mathbb{S}_1^{n+2} \\
 \mathbb{S}^n \ar[ru]_-{\psi_{f,\alpha}}   
} 
\]
In particular, the immersion $\psi$ is an embedding.
Observe that when $\psi(\Sigma)\subset\Lambda_{\alpha,-}^+$ then  $\displaystyle{f=u\circ \Psi^{-1}:\mathbb{S}^{n}\to \left(0,\frac{\sqrt{1-\alpha^2}}{\alpha}\right)}$ (as in Example \ref{ex:luisA4}) and when $\psi(\Sigma)\subset\Lambda_{\alpha,+}^+$ then  $\displaystyle{f=u\circ \Psi^{-1}:\mathbb{S}^{n}\to \left(\frac{\sqrt{1-\alpha^2}}{\alpha},+\infty\right)}$ (as in Example \ref{ex:luisA5}).
\end{theorem}
\begin{proof}
We consider $\Psi$ as in the proof of Theorem \ref{thm:GCone} and Proposition \ref{prop:luisA6} and recall that in this situation 
$\Psi:(\Sigma^n,\langle,\rangle_\Sigma)\to(\mathbb{S}^{n},\langle,\rangle_{\mathbb{S}^n})$ is a conformal diffeomorphism with 
\[
\Psi^*(\langle \,,\rangle_{\mathbb{S}^n}) = 
\frac{1}{R(u)^2} \langle \,,  \rangle_{\Sigma},
\]
with $u=\psi_1$. Let $f= u \circ \Psi^{-1}$. Then $f\circ \Psi = u$ and 
\begin{align*}
(\psi_{f, \alpha} \circ \Psi)(p) &=
\left( f(\Psi(p)), R(f(\Psi(p))) \, \Psi(p), \alpha+ \sqrt{1-\alpha^2} f(\Psi(p)) \right)\\
&= (u(p) , R(u) \Psi(p), \alpha+ \sqrt{1-\alpha^2} u(p))\\
&=
( u(p), \psi_2(p), \ldots, \psi_{n+2}(p) ,  \alpha+ \sqrt{1-\alpha^2} u(p)) \\
&= \psi(p).
\end{align*}
\end{proof}

\section*{Acknowledgements}
J. Mel\'endez acknowledges the support of PEAPDI 2025 grant, CBI-UAMI and SECIHTI SNII 165260. He wants to thank the hospitality of D. A. Solis and M. Navarro while visiting them at UADY. He is also very grateful to L.J. Al\'{\i}as and the Department of Mathematics of the University of Murcia for the hospitality and support. M. Navarro acknowledges the support of SECIHTI SNII 25997. D. A. Solis acknowledges the support of SECIHTI SNII 38368.

L.J. Al\'{\i}as and Didier A. Solis are  partially supported 
by the grant PID2021-124157NB-I00 funded by 
MCIN/AEI/10.13039/501100011033/ `ERDF A way of making Europe',
Spain.

\bibliographystyle{amsplain}

\section{Appendix}

The following technical lemma establishes a relation between the scalar curvatures of conformally related metrics. For the sake of completeness and for the reader's convenience, we include a detailed proof.

\begin{lemma} 
\label{lem:conformal}
Let $\Sigma^n$ be a Riemannian manifold endowed with conformally related metrics $\langle \, , \rangle$ and $\widetilde{\langle \, , \rangle} = \lambda^2 \langle \, , \rangle$. Let $E_i$ and $E_j $ be an $ \langle \, , \rangle$-orthonormal vectors. Then their sectional curvatures $\widetilde{K}$ and $K$, of the plane $\sigma$ generated by $E_i$ and $E_j$, are related by
\begin{align}
\label{eq:sect_curvatures}
\lambda^4 \widetilde{K}(\sigma)=&
 \lambda^2 K(\sigma)
 + 2 \left( (E_i(\lambda))^2 + (E_j(\lambda))^2\right) 
 - \lambda ( 
 \mathop\mathrm{Hess} \lambda (E_i,E_i)+ 
 \mathop\mathrm{Hess} \lambda (E_j,E_j)) \nonumber \\
 &- \| \nabla \lambda\|^2        
\end{align}
where  $\|  \cdot \|$, $\nabla$ and $\mathop\mathrm{Hess}$ refer to the norm, gradient and Hessian for the metric  $\langle \, , \rangle$.  In particular, the scalar curvatures are related by
\[
\lambda^2 \widetilde{\mathop\mathrm{Scal}} = \mathop\mathrm{Scal} -2(n-1) \frac{\Delta \lambda}{\lambda}- (n-1)(n-4) \frac{\|\nabla \lambda\|^2}{ \lambda^2}
\]
where the Laplacian operator $\Delta$ is with respect to the metric  $\langle \, , \rangle$. 
\end{lemma}

\begin{proof}
We use the formula for conformal metrics (see Proposition 3.9 in \cite{sakai})
\begin{equation*}
\overline{R}=\lambda^2 R + \left(
\mathop\mathrm{Hess} (\log \lambda)  -  d(\log \lambda)  \otimes d(\log \lambda)  +\frac{1}{2}\| \nabla \log \lambda\|^2\langle \, , \rangle \right)\odot \overline{\langle \, , \rangle}
\end{equation*}
where $R$ and $\overline{R}$ denote the curvature tensors of $\langle \,,\rangle$
and $\lambda^2 \langle \,,\rangle$ respectively\footnote{Recall that, for the tensors $h$ ,$k$ of type $(0,2)$, the Kulkarni-Nomizu product $\odot$ is defined by 
\[
(h\odot k )(x,y,z,w)= h(x,z)k(y,w) + h(y,w) k(x,z) - h(x,w) k (y,z) -h(y,z) k(x,w).\]
}
Note that if we set 
\[
h:=\mathop\mathrm{Hess} (\log \lambda)  
-  d(\log \lambda)  \otimes d(\log \lambda)  
+\frac{1}{2}\| \nabla \log \lambda\|^2 \, \langle \, , \rangle
\]
 then
\[
(h\odot \widetilde{\langle \, , \rangle}) (E_i,E_j,E_j,E_i) = -\lambda^2 ( h(E_i,E_i) + h(E_j,E_j)).
\]
First, we  consider the terms of $h(E_i,E_i)$ separately, noting that
\[
\mathop\mathrm{Hess} (\log \lambda) (E_i,E_i) = \langle \nabla_{E_i} \nabla \log \lambda, E_i \rangle =\frac{1}{\lambda} \mathop\mathrm{Hess} \lambda(E_i,E_i)-\frac{1}{\lambda^2} (E_i(\lambda))^2
\]
and 
\[
 d(\log \lambda)  \otimes d(\log \lambda) (E_i,E_i) = \frac{(E_i(\lambda))^2}{\lambda^2},
 \qquad \quad
\frac{1}{2} \| \nabla \log \lambda\|^2= 
\frac{1}{2\lambda^2} \|\nabla \lambda \|^2.
\]
On the other hand, since the sectional curvatures satisfy
\[
K = \langle R(E_i,E_j)E_j, E_i \rangle
\qquad \text{and}
\qquad 
\lambda^4 \widetilde{K} = \langle \overline{R}(E_i,E_j)E_j, E_i \rangle
\]
we have, after a straightforward calculation
\begin{align*}
\lambda^4 \widetilde{K}  
=\lambda^2 K+2( (E_i(\lambda))^2+(E_j(\lambda))^2) -\lambda\left( \mathop\mathrm{Hess} \lambda(E_i,E_i)+\mathop\mathrm{Hess}\lambda(E_j,E_j) \right) - \|\nabla \lambda\|^2.   
\end{align*}
\end{proof}

\begin{remark}
For $n=2$, we obtain
\[
\lambda^2 \widetilde{\mathop\mathrm{Scal}} = \mathop\mathrm{Scal} -2 \frac{\Delta \lambda}{\lambda}+ 2 \frac{\|\nabla \lambda\|^2}{ \lambda^2} = \mathop\mathrm{Scal} - 2\Delta \log \lambda.
\] 

Since $\mathop\mathrm{Scal} =2K$,  the Gaussian curvatures $K$ and $\widetilde{K}$ are related by
\begin{equation}
    \label{eq:gauss_conformal}
    \lambda^2 \widetilde{K} = K- \Delta \log \lambda.
\end{equation}
\end{remark}

\end{document}